\documentclass[a4paper,11pt]{article}

\usepackage[english]{babel}
\usepackage{theorem,epsfig,wrapfig,array}
\usepackage{varioref,float,amssymb,amsmath,float}

\newcommand{\me}{\mbox{\rm\tiny 1/2}}
\newcommand{\D}{\protect\displaystyle}
\newcommand{\T}{\protect\textstyle}

\newcommand{\eps}{\varepsilon}

\newcommand{\vphi}{\varphi}
\newcommand{\lbd}{\lambda}

\newtheorem{theorem}{Theorem}
\newtheorem{lemma}[theorem]{Lemma}
\newtheorem{corollary}[theorem]{Corollary}
\newtheorem{prop}[theorem]{Proposition}

\newtheorem{remark}[theorem]{Remark}

\def\N{{\mathord{\rm I\mkern-3.6mu N}}}

\def\R{{\mathord{\rm I\mkern-3.6mu R}}}

\begin{document}
\setcounter{footnote}{1}
\title{Mean value iterations for nonlinear elliptic Cauchy problems}

\author{P.\;K\"ugler%
\footnote{Supported by FWF project F--1308 within
Spezialforschungsbereich 13, email: kuegler@indmath.uni-linz.ac.at}
\, and\, A.\;Leit\~ao$^\dag$%
\footnote{On leave from Department of Mathematics, Federal University of 
St.\,Catarina, P.O. Box 476, 88010-970 Florian\'opolis, Brazil} \\
Institut f\"ur Industriemathematik, \\ Johannes Kepler Universit\"at,
A--4040 Linz, Austria}

\date{}
\maketitle


\begin{abstract}
We investigate the Cauchy problem for a class of nonlinear elliptic operators 
with $C^\infty$--coefficients at a regular set $\Omega \subset \R^n$. The 
Cauchy data are given at a manifold $\Gamma \subset \partial\Omega$ and 
our goal is to reconstruct the trace of the $H^1(\Omega)$ solution of a 
nonlinear elliptic equation at $\partial \Omega / \Gamma$.
We propose two iterative methods based on the segmenting Mann iteration 
applied to fixed point equations, which are closely related to the original 
problem. The first approach consists in obtaining a corresponding linear 
Cauchy problem and analyzing a linear fixed point equation; a 
convergence proof is given and convergence rates are obtained. On the 
second approach a nonlinear fixed point equation is considered and a 
fully nonlinear iterative method is investigated; some preliminary 
convergence results are proven and a numerical analysis is provided.
\end{abstract}

\setcounter{footnote}{0}
\pagestyle{plain}
%
%
%
\section{Introduction} \label{sec:introd}

The main results discussed in this paper are: the solution of a nonlinear 
elliptic Cauchy problem is written as the solution of fixed point equations; 
mean value iterations are used to approximate the solution of these fixed 
point equations. We follow two different approaches: in the first one we 
use a nonlinear transformation in order to obtain a linear Cauchy problem 
and a corresponding linear fixed point equation. We give a convergence 
proof and also prove some convergence rates. In the second approach we 
deduce a nonlinear fixed point equation for the solution of the Cauchy 
problem and a fully nonlinear iterative method is considered. We prove 
preliminary convergence results and analyze several numerical examples.

The fixed point equations corresponding to the nonlinear Cauchy problem 
are obtained in Section~\ref{sec:nlecp-fpgl}. In order to construct the 
fixed point operators, two main steps are required: in Section~%
\ref{sec:el-cau-prbl} we obtain a particular version of the 
Cauchy--Kowalewskaia theorem for the nonlinear elliptic Cauchy problem 
of interest; in Section~\ref{sec:el-gem-prbl} we prove existence and 
uniqueness of solutions for mixed boundary value problems associated 
with the same nonlinear differential operator. It is worth mentioning 
that, in the linear case, a fixed point equation for the Cauchy problem 
is analyzed in \cite{EnLe}.

The formulation of the Mann iteration is discussed in Section~\ref{sec:mann} 
(see \cite{Ma}). We also analyze some extensions of the original result 
obtained by W.\;Mann, among these, a variant introduced by C.\;Groetsch, 
called segmenting Mann iteration (see \cite{Gr1}, \cite{EnSc}). All these 
iterative methods aim to approximate the solutions of fixed point equations.

In Section~\ref{sec:it_cp} we formulate the mean value iterations for 
the fixed point equations obtained in Section~\ref{sec:nlecp-fpgl}. Some 
analytical results (convergence, rates, \dots) are discussed. In Section%
~\ref{sec:numeric} we analyze numerically the fully nonlinear iteration.
%
%
%
\section{Elliptic Cauchy problems} \label{sec:el-cau-prbl}

Let $\Omega \subset \R^n$ be an open bounded set and $\Gamma \subset \partial 
\Omega$ a given manifold. Given the function $q: \R \to \R+$, we denote by 
$P$ a second order elliptic operator of the form
\begin{equation} \label{gl:oper-P-def}
P(u) \ = \ - \nabla \cdot ( q(u) \nabla u )
\end{equation}
defined in $\Omega$. We denote by {\em nonlinear elliptic Cauchy problem} 
the following (time independent) initial value problem for the operator 
$P$ \\[2ex]
$ (CP) \hskip2cm
  \left\{ \begin{array}{rl}
      P u        = h\, ,&\!\! \mbox{in } \Omega \\
      u          = f\, ,&\!\! \mbox{at } \Gamma \\
      q(u) u_\nu = g\, ,&\!\! \mbox{at } \Gamma
  \end{array} \right. $ \\[2ex]
where the given functions $f, g: \Gamma \to \R$ are called {\em Cauchy 
data} and the right hand side of the differential equation is a function 
$h: \Omega \to \R$. The problem we want to solve is to evaluate the trace 
of the solution of such an initial value problem at the part of the 
boundary where no data is prescribed, i.e. at $\partial\Omega \backslash 
\Gamma$. As a solution of the Cauchy problem (CP) we consider a 
$H^1(\Omega)$--distribution, which solves the weak formulation of the 
elliptic equation in $\Omega$ and also satisfies the Cauchy data at 
$\Gamma$ in the sense of the trace operator.

It is well known that linear elliptic Cauchy problems are not well posed 
in the sense of Hadamard.%
\footnote{For a formal definition of {\em well posed} problems, see e.g. 
\cite{EHN}.}
A famous example given by Hadamard himself (see \cite{Le} and the 
references therein) shows that we cannot have continuous dependence 
on the data. Also existence of solutions for arbitrary Cauchy data 
$(f,g)$ cannot be assured,%
\footnote{The Cauchy data $(f,g)$ is called {\em consistent} if the 
corresponding problem (CP) has a solution. Otherwise $(f,g)$ is called 
{\em inconsistent} Cauchy data.}
as shows a simple argumentation with the Schwartz reflection principle 
(see \cite{Tr}).

In Section~\ref{ssec:cau-kow-lin} we extend the Holmgren theorem to the 
$H^1$--context, proving uniqueness of solutions in weak sense for linear 
elliptic Cauchy problems. The next step, described in Section~%
\ref{ssec:cau-kow-nlin}, is to extend this weak uniqueness result to 
the nonlinear elliptic Cauchy problem (CP).

\subsection{Uniqueness: the linear case} \label{ssec:cau-kow-lin}

In this section we briefly recall an uniqueness result for elliptic 
operators. This result extends, in the elliptic case, the well known 
Holmgren theorem (see, e.g. \cite{Fo}) for operators with analytic 
coefficients and Cauchy data given at a non characteristic manifold.

A well known theorem by A.L.\;Cauchy and S.\;Kowalewskaia yields uniqueness 
of a locally analytic solution to a Cauchy problem, where the differential 
operator has analytic coefficients and the data are analytic on a analytic 
non-characteristic manifold. The Holmgren theorem guarantees, that for 
linear differential equations no other non-analytic solution exists, even 
if one renounces the analyticity of the Cauchy data.

In order to treat weak solutions of Cauchy problems, we take 
advantage of some regularity results. Essentially one needs to know 
that, given a strongly elliptic linear operator of second order $L$ with 
$C^\infty(\Omega)$--coefficients and a distribution $h \in 
H^k_{\rm loc}(\Omega)$, then a solution of $Lu = h$ has to satisfy 
$u \in C^j(\Omega)$ for $j < k+2 - (n/2)$. In particular, if $h \in 
C^\infty(\Omega)$, then $u$ also belongs to $C^\infty(\Omega)$.

Next we state the desired uniqueness result for weak solutions of the 
linear Cauchy problem. For the convenience of the reader a sketch of 
the proof is given.

\begin{theorem} \label{satz:H1-eindeut-lin}
Let $\Omega$ be an open bounded simply connected subdomain of $\R^n$ 
with analytic boundary, and let $\Gamma$ be an open, connected part of 
$\partial\Omega$. Let the operator $L$ be defined as above. Then, the 
Cauchy problem
$$ \left\{  \begin{array}{rl}
     Lu    = h\, ,&\!\! \mbox{in}\ \Omega \\
     u     = f\, ,&\!\! \mbox{at}\ \Gamma \\
     u_\nu = g\, ,&\!\! \mbox{at}\ \Gamma
   \end{array} \right.  $$
for $h \in L^2(\Omega)$, $f \in H^{\me}(\Gamma)$ and $g \in 
H^{\me}_{00}(\Gamma)'$ has at most one solution in $H^1(\Omega)$.%
\footnote{The Sobolev spaces are defined as in \cite{DaLi}; see also 
\cite{Ad}.}
\end{theorem}

\begin{sketch}
If $u_1$ and $u_2$ are $H^1$ --solutions of the Cauchy problem, then 
$u = u_1 - u_2$ solves the following problem:
$$ \left\{  \begin{array}{ll}
     Lu = 0       \, ,&\!\! \mbox{in}\ \Omega \\
     u = u_\nu = 0\, ,&\!\! \mbox{at}\ \Gamma
   \end{array} \right.  $$
Because of $Lu = 0$, our regularity result (with $h = 0$) yields that 
$u \in C^\infty(\Omega)$ holds. Hence, $u$ is a classical solution to 
the Cauchy problem with homogeneous data on the manifold $\Gamma$.

The operator $L$ is strongly elliptic and both $\Gamma$ and $\partial 
\Omega \backslash \Gamma$ are non--characteristic manifolds. Hence, it 
is possible to find a family of analytic manifolds $\Gamma_\lambda$, 
$0 \leq \lambda \leq 1$ such that $\Gamma_0 = \Gamma$, $\Gamma_1 = 
\partial\Omega \backslash \Gamma$ and all $\Gamma_\lambda$ share the 
same endpoints. If we use the family $\Gamma_\lambda$ from the proof 
of Holmgren theorem, we can conclude that the function $u$ vanishes 
identically on $\Omega$.
\end{sketch}

\subsection{Uniqueness: the nonlinear case} \label{ssec:cau-kow-nlin}

In this Section we prove for the nonlinear elliptic Cauchy problem (CP) a 
result analogous to the one stated in Theorem~\ref{satz:H1-eindeut-lin}. 
The argumentation in the next theorem follows the lines of 
\cite[Theorem~2.3.2]{Ku}

\begin{theorem} \label{satz:H1-eindeut-nlin}
Let $\Omega$ be an open bounded simply connected subset of $\R^n$ with 
analytical boundary, $\Gamma \subset \partial\Omega$ an open connected 
manifold, the operator $P$ defined as in \eqref{gl:oper-P-def}, where 
$q: \R \to [q_{min}, q_{max}] \subset (0,\infty)$ is a $C^\infty$--function. 
Then the nonlinear Cauchy problem (CP) has for each pair of data $(f,g) \in 
H^{\me}(\Gamma) \times H^{\me}_{00}(\Gamma)'$ and $h \in L^2(\Omega)$ at 
most one $H^1(\Omega)$ solution.
\end{theorem}

\begin{proof}
Notice that $P u = F[u]$, where
$$ F[u] \ = \ F\Big(u, \frac{\partial u}{\partial x_i},
            \frac{\partial^2 u}{\partial x_i \partial x_j}\Big)
       \ := \ q'(u) (\nabla u)^2 \, + \, q(u) \Delta u . $$
The function $F = F(s,p,R)$, with $s \in \R$, $p = (p_1, \ldots, p_n)$, $R = 
(r_{ij})_{i,j=1}^n$ is continuously differentiable because of the assumption 
on $q$. Furhtermore, the operator $P$ is elliptic due to the strict positivity 
of $q$. Now let $u_1$, $u_2$ be two $H^1$--solutions of (CP). Notice that
\begin{equation} \label{gl:h1-eind-1}
F\Big(u_1, \frac{\partial u_1}{\partial x_i},
  \frac{\partial^2 u_1}{\partial x_i \partial x_j}\Big) \, - \,
F\Big(u_2, \frac{\partial u_2}{\partial x_i},
  \frac{\partial^2 u_2}{\partial x_i \partial x_j}\Big) \ = \ 0 .
\end{equation}
Defining $v := u_1 - u_2$, it follows from \eqref{gl:h1-eind-1} together 
with the mean-value theorem (for functions of several variables) that
\begin{equation} \label{gl:h1-eind-2}
\sum\limits_{i,j=1}^n \Big( \frac{\partial F}{\partial r_{ij}}
    \Big)_{\!\theta} \ \frac{\partial^2 v}{\partial x_i \partial x_j} \ + \ 
\sum\limits_{i=1}^n \Big( \frac{\partial F}{\partial p_i}
    \Big)_{\!\theta} \ \frac{\partial v}{\partial x_i} \ + \ 
\Big( \frac{\partial F}{\partial s} \Big)_{\!\theta} \ v \ = \ 0 ,
\end{equation}
for (different) $\theta \in (0,1)$ with
\begin{eqnarray*}
\Big( \frac{\partial F}{\partial r_{ij}} \Big)_{\!\theta} & = & 
q (\theta u_1 + (1-\theta) u_2) \delta_{ij} , \\
\Big( \frac{\partial F}{\partial p_i} \Big)_{\!\theta} & = &
2q' (\theta u_1 + (1-\theta) u_2) \, (\theta \frac{\partial u_1}{\partial x_i}
 + (1-\theta) \frac{\partial u_2}{\partial x_i} ) , \\
\Big( \frac{\partial F}{\partial s} \Big)_{\!\theta} & = &
q'' (\theta u_1 + (1-\theta) u_2)\, (\theta \nabla u_1 + (1-\theta)
\nabla u_2)^2 \\
& & + \ q' (\theta u_1 + (1-\theta) u_2) \,
      (\theta \Delta u_1 + (1-\theta) \Delta u_2)
\end{eqnarray*}
(note that $\Delta u_i \in L^2$, since $F[u_i] = h$). Thus, $v$ is a 
$H^1(\Omega)$--solution of the linear elliptic differential equation 
\eqref{gl:h1-eind-2} and further $v$ satisfies $v|_\Gamma = 0$. Now 
it follows from the identity $q(u_1) u_1 - q(u_2) u_2 = 0$ (again 
argumenting with the mean-value theorem)
\begin{displaymath}
q' (\theta u_1 + (1-\theta) u_2)(\theta (u_1)_\nu
 + (1-\theta) (u_2)_\nu )v + 
q (\theta u_1 + (1-\theta) u_2) v_\nu = 0 \;\;
\mbox{on} \;\; \Gamma
\end{displaymath}
for another $\theta \in (0,1)$. Finally, the positivity of $q$ yields 
$v_\nu |_\Gamma = 0$. From Theorem~\ref{satz:H1-eindeut-lin} follows 
now $v \equiv 0$, proving the assertion.
\end{proof}

\begin{remark}
There is an alternative way to obtain the uniqueness result described 
in Theorem~\ref{satz:H1-eindeut-nlin}. Namely, using the nonlinear 
transformation $Q$ in Section~\ref{sec:nlecp-fpgl} and applying 
Theorem~\ref{satz:H1-eindeut-lin} to the corresponding linear Cauchy 
problem. The choice of the approach presented in the proof of 
Theorem~\ref{satz:H1-eindeut-nlin} is motivated by the nonlinear 
nature of the operators associated to problem (CP).
\end{remark}
%
%
%
\section{Elliptic mixed boundary value problems} \label{sec:el-gem-prbl}

In this section we prove existence and uniqueness of solutions for mixed 
boundary value problems modelled by elliptic differential operators. We start 
presenting some results concerning the linear case in Section~%
\ref{ssec:el-gem-prbl-l}. In Section~\ref{ssec:el-gem-prbl-nl} we extend 
this results for the nonlinear operator defined in \eqref{gl:oper-P-def}.

\subsection{The linear case} \label{ssec:el-gem-prbl-l}

Initially we introduce some useful Sobolev spaces. Let $\Omega \in \R^n$ 
be an open bounded regular set%
\footnote{We mean $\Omega$ is locally at one side of $\partial\Omega$.} 
with $C^\infty$--boundary $\partial\Omega$, which is split in 
$\partial\Omega = \overline{\Gamma_1} \cup \overline{\Gamma_2}$, the subsets 
$\Gamma_j$ being open connected and satisfying $\Gamma_1 \cap \Gamma_2 = 
\emptyset$. We define
$$ H^s_0(\Omega\cup\Gamma) \ := \ \overline{C^\infty_0(\Omega\cup\Gamma)}^
   {\|\cdot\|_{s;\Omega}},  $$
where $\|\cdot\|_{s;\Omega}$ denotes the usual Sobolev s--norm on $\Omega$.%
\footnote{Analogously one can define
$$   H^s(\Omega) := \overline{C^\infty(\overline{\Omega})}^
                                          {\|\cdot\|_{s;\Omega}} ,\
    H^s_0(\Omega) := \overline{C^\infty_0(\Omega)}^{\|\cdot\|_{s;\Omega}} . $$
}

Let $L$ be the linear elliptic operator defined in section~%
\ref{ssec:cau-kow-lin} and denote by $a(\cdot,\cdot)$ the 
corresponding bilinear form. For the analysis of mixed problems 
we need a {\em Poincar\'e type} inequality on the space 
$H^1_0(\Omega\cup\Gamma_j)$, i.e.
\begin{equation} \label{eq:poincarre}
  \|u\|_{L^2(\Omega)} \ \leq \ c \, \| a(u,u) \| ,\ \ 
  u \in H^1_0(\Omega\cup\Gamma_j)
\end{equation}
(see, e.g. \cite{Tr}). We consider the following mixed problem problem 
at $\Omega$: given the functions $f \in H^{\me}(\Gamma_1)$ and $g \in 
H^{\me}_{00}(\Gamma_2)'$, find a $H^1$--solution of \\[2ex]
$ (LMP) \hskip2cm
  \left\{ \begin{array}{rl}
     L u   = 0\, ,&\!\! \mbox{in}\ \Omega \\
     u     = f\, ,&\!\! \mbox{at}\ \Gamma_1 \\
     u_\nu = g\, ,&\!\! \mbox{at}\ \Gamma_2
  \end{array} \right. . \hfill $ \\[2ex]
Existence, uniqueness and continuous dependency of the data for (LMP) 
are given by a lemma of Lax--Milgram type, which states that for every 
pair of Cauchy data, problem (LMP) has a unique solution $u \in 
H^1(\Omega)$ and, further,
\begin{equation} \label{gl:stetig-abhang}
    \|u\|_{H^1(\Omega)} \ \leq \ c \left( \|f\|_{H^{\me}(\Gamma_1)}\ +\
                                      \|g\|_{H^{\me}_{00}(\Gamma_2)'} \right).
\end{equation}
This is a standard result and we omit the proof.

\subsection{The nonlinear case} \label{ssec:el-gem-prbl-nl}

Now, we extend the results of section~\ref{ssec:el-gem-prbl-l} to 
the following nonlinear mixed boundary value problem: find a 
$H^1$--solution of \\[2ex]
$ (MP) \hskip2cm
  \left\{ \begin{array}{rl}
    P(u)       = h\, ,&\!\! \mbox{in}\ \Omega \\
    u          = f\, ,&\!\! \mbox{at}\ \Gamma_1 \\
    q(u) u_\nu = g\, ,&\!\! \mbox{at}\ \Gamma_2
  \end{array} \right. \hfill $ \\[2ex]
where the coefficient $q \in H^1(\R)$, the right hand side 
$h \in L^2(\Omega)$, and the boundary conditions $f \in H^{\me}(\Gamma_1)$, 
$g \in H^{\me}_{00}(\Gamma_2)'$ are given functions.

Before proving the main result of this section, we present an existence 
(and uniqueness) theorem for abstract operator equations concerning 
quasi-monotone operators.

\begin{lemma} \label{lem:quasi-mon}
Let the space ${\cal V}$ satisfy $H^1_0(\Omega) \subset {\cal V} \subset 
H^1(\Omega)$ and assume that ${\cal V} \hookrightarrow L^2(\Omega)$ is 
compact. Let the operator $\bar{A}:{\cal V} \rightarrow {\cal V}'$ be 
given by $\bar{A}u = A(u,u)$, with $A:{\cal V} \times {\cal V} 
\rightarrow {\cal V}'$ defined by
\begin{equation}
\langle A(u,v), w \rangle_{\cal V} \ := \ 
\langle A_1(u,v), w \rangle_{\cal V} + \langle A_0u, w \rangle_{\cal V}
\end{equation}
where
\begin{eqnarray}
\langle A_1(u,v), w \rangle_{\cal V} & := &
\int_\Omega \Big\{ \sum\limits_{j = 1}^n a_j(x,u(x),\nabla v(x))\,
            \frac{\partial}{\partial x_j} w(x) \Big\}\, dx , \\
\langle A_0 u, w \rangle_{\cal V} & := &
\int_\Omega a_0(x,u(x),\nabla u(x)) \, w(x)\, dx
\end{eqnarray}
for all $u, v, w \in {\cal V}$. Here, the coefficient functions $a_j: 
\Omega \times \R \times \R^n \to \R$, $0 \le j \le n$ are supposed 
to satisfy
\begin{enumerate}
\item $a_j(x,\eta,\xi)$ is measurable in $x \in \Omega$ and continuous in 
$(\eta,\xi) \in \R \times \R^n$;
\item $|a_j(x,\eta,\xi)| \le c \left(k(x) + |\eta| + \|\xi\| \right)$ for 
almost all $x \in \Omega$, $\eta \in \R$, $\xi \in \R^n$, 
where $k \in L^2(\Omega)$;
\item $\sum\limits_{j = 1}^n \big( a_j(x,\eta,\xi)-a_j(x,\eta,\tilde{\xi}) 
\big) \big( \xi_j - \tilde{\xi}_j \big) > 0$ for almost all $x \in \Omega$, 
$\eta \in \R$ and $\xi \in \R^n / \{\tilde{\xi}\}$;
\item $\frac{\sum\limits_{j=1}^n a_j(x,\eta,\xi)\xi_j}{\|\xi\|+\|\xi\|} 
\to \infty$ as $\|\xi\| \rightarrow \infty$, uniformly for $\eta$ bounded at 
almost all $x \in \Omega$.
\end{enumerate}
Furthermore, let ${\cal F}$ belong to ${\cal V}'$. Then, if $\bar{A}$ is 
${\cal V}$-coercive and bounded, the equation
\begin{equation}
\bar{A}u = {\cal F}
\end{equation}
has a unique solution in ${\cal V}$.
\end{lemma}

\begin{sketch}
This lemma follows from an analog result concerning monotone operators by 
Browder--Vishik (see \cite[Proposition~5.1]{Sh}, page 60). One has 
just to relax the monotony assumptions to the class of quasi-monotone 
operators, as shown in \cite[Section~II.6]{Sh}, on page 74 and subsequent.
\end{sketch}

The next result guarantees existence and uniqueness of solution for 
problem (MP) and is the nonlinear analogous to the result stated 
in section~\ref{ssec:el-gem-prbl-l}.

\begin{theorem} \label{lem:ex-eind-gemischt-nl}
Let the function $q$ be such that $0 < q_{min} \le q \le q_{max} < \infty$ 
and $h \in L^2(\Omega)$ be given. Then for every pair of data $(f,g) \in
H^{\me}(\Gamma_1) \times H^{\me}_{00}(\Gamma_2)'$ the problem (MP) has a 
unique solution $u \in H^1(\Omega)$. 
\end{theorem}

\begin{proof}
We consider first the homogeneous case $f = 0$. Setting ${\cal V} = 
H^1_0(\Omega\cup\Gamma_2)$ we derive by integration by parts the following 
variational formulation for (MP):
\begin{equation}\label{variation}
  \begin{array}{l}
     \hskip-1cm \mbox{Find $u_0 \in {\cal V}$ such that} \\[1ex]
     \hskip0.5cm \D\int_\Omega q(u_0)\nabla u_0 \cdot \nabla v\, dx \ = \ 
     \int_\Omega h(x) v\, dx + \int_{\Gamma_2} g v\, d\Gamma_2 \\[2ex]
     \hskip-1cm \mbox{holds for all} v \in {\cal V}.
  \end{array} \hfill
\end{equation}
This is equivalent to solve the operator equation
\begin{equation} \label{gl:nl-gp}
\bar{A}u_0 = {\cal F},
\end{equation}
where $\bar{A}$ is given by $a_0(x,\eta,\xi) = 0$ and $a_j(x,\eta,\xi) = 
q(\eta)\xi_j$ (see Lemma~\ref{lem:quasi-mon}) and ${\cal F} \in {\cal V}'$ is defined by
\begin{equation}
\langle{\cal F},v\rangle_{\cal V} \ = \ 
\int_\Omega h(x) v\, dx + \int_{\Gamma_2} g v\, d\Gamma_2 \ 
\mbox{ for all } \ v, w \in {\cal V}.
\end{equation}
We easily see that the conditions 1--4 of Lemma~\ref{lem:quasi-mon} are 
satisfied. Indeed,
\begin{enumerate}
\item $ q(\eta)\xi_j$ is continuous in $(\eta,\xi) \in \R \times \R^n$;
\item $|q(\eta)\xi_j| \le q_{\max} \|\xi\|$;
\item $\sum\limits_{j = 1}^n \left(q(\eta)\xi_j-q(\eta)\tilde{\xi}_j \right) 
\left(\xi_j - \tilde{\xi}_j \right) \ge 
q_{\min} \sum\limits_{j = 1}^n \left(\xi_j - \tilde{\xi}_j \right)^2 > 0$ 
for all $\eta \in \R$, $\xi \neq \tilde{\xi} \in \R^n$;
\item $\frac{\sum\limits_{j = 1}^n q(\eta)\xi_j\xi_j}{\|\xi\|+\|\xi\|} \ge 
\frac{q_{\min}\|\xi\|^2}{2\|\xi\|} \rightarrow \infty$ as $\|\xi\| \to \infty$.
\end{enumerate}
Since the coercitivity of $\bar{A}$ follows from the Poincar\'e inequality 
and the boundedness is given by the bound $q_{\max}$, the existence of a 
unique $u_0 \in {\cal V}$ can be guaranteed.

Next we consider the non-homogeneous case. By the trace theorem, we obtain a 
function $\tilde{f} \in H^1(\Omega)$ such that $\tilde{f}|_{\Gamma_1} = f$. 
Defining the closed convex non-empty affine subspace $K := \tilde{f} + H_0^1 
(\Omega\cup\Gamma_2)$, it follows from a result by Brezis (see e.g. 
\cite[Theorem~2.3]{Sh}, on page 42) that equation \eqref{gl:nl-gp} has a 
unique solution in $K$. This solution clearly satisfies both the differential 
equation $P(u) = h$ and the mixed boundary conditions.
\end{proof}
%
%
%
\section{Cauchy problems and fixed point equations} \label{sec:nlecp-fpgl}

In this section we characterize the solution of (CP) as solution of 
certain fixed point equations. We define $\Gamma_1 := \Gamma$, $\Gamma_2 
:= \partial\Omega \backslash \Gamma$, such that $\Gamma_1\cap\Gamma_2 = 
\emptyset$ and $\overline{\Gamma_1\cup\Gamma_2} = \partial\Omega$. 
Further, we assume that the coefficient $q$ of the the second order 
elliptic operator $P$ defined in \eqref{gl:oper-P-def} satisfies
\begin{itemize}
\item[\it A1)] $q \in C^\infty(\R)$;
\item[\it A2)] $q(t) \in [q_{min},q_{max}]$ for all $t \in \R$,\, where\, 
$0 < q_{min} < q_{max} < \infty$.
\end{itemize}
Given the Cauchy data $(f,g) \in H^{\me}(\Gamma_1) \times H^{\me}_{00} 
(\Gamma_1)'$ and $h \in L^2(\Omega)$, we assume that there exists a 
$H^1$--solution of problem \\[2ex]
$ (CP) \hskip2cm
   \left\{ \begin{array}{rl}
      P u        = h\, ,&\!\! \mbox{in}\ \Omega \\
      u          = f\, ,&\!\! \mbox{at}\ \Gamma_1 \\
      q(u) u_\nu = g\, ,&\!\! \mbox{at}\ \Gamma_1
   \end{array} \right. $ \\[2ex]
and we are mainly interested in the determination of the Neumann trace
\begin{equation} \label{gl:sol-nlcp}
\bar{\vphi} \, := \, q(u) u_{\nu|_{\Gamma_2}} \in H^{\me}_{00}(\Gamma_2)' .
\end{equation}
Notice that, once $\bar{\vphi}$ is known, the solution of (CP) can be 
determined as the solution of a well posed mixed boundary value problem, 
namely
$$ \left\{ \begin{array}{rl}
      P u        = h          \, ,&\!\! \mbox{in}\ \Omega \\
      u          = f          \, ,&\!\! \mbox{at}\ \Gamma_1 \\
      q(u) u_\nu = \bar{\vphi}\, ,&\!\! \mbox{at}\ \Gamma_2.
  \end{array} \right. $$

In order to obtain the fixed point equations for $\bar{\vphi}$, we follow 
two separate approaches, which we describe next:
\medskip

\noindent {\bf First approach:} Transformation of (CP) into a linear Cauchy problem
\medskip

\noindent  The first step is to introduce the real function
$$  Q(t) := \int_0^t q(s)\, ds . $$
Given $u \in H^1(\Omega)$, the function $U := Q(u)$ is also in $H^1(\Omega)$ 
and satisfies
$$  \Delta U = \nabla \cdot (\nabla Q(u))
             = \nabla \cdot (q(u) \nabla u) = P(u) . $$
Further, $U_\nu = q(u) u_\nu$ holds at $\partial \Omega$. Thus, if $u$ 
is the solution of problem (CP), the function $U$ solves the linear 
Cauchy problem \\[2ex]
$ (LCP) \hskip2cm
  \left\{ \begin{array}{rll}
      \Delta U = &\!\! h   \, ,&\!\! \mbox{in}\ \Omega \\
             U = &\!\! Q(f)\, ,&\!\! \mbox{at}\ \Gamma_1 \\
         U_\nu = &\!\! g   \, ,&\!\! \mbox{at}\ \Gamma_1
  \end{array} \right. $ \\[2ex]
Reciprocally, if problem (LCP) is consistent for some data $(Q(f),g)$ with 
solution $U$, it follows from $Q' = q > 0$ (we use assumption {\it A2)}), 
that $u := Q^{-1}(U)$ is well defined in $H^1(\Omega)$ and solves problem 
(CP).

Thus, it is enough to find the solution of the (consistent) Cauchy problem 
(LCP). Since this is a linear problem, a fixed point equation based on the 
composition of adequate mixed boundary value problems is known from the 
literature (see, e.g., \cite{Le} or \cite{KMF}). This fixed point equation 
is obtained as follows:

\begin{itemize}
\item[\it 1)] Let $\bar{\vphi}$ be given by \eqref{gl:sol-nlcp}. Notice 
that $U$, the solution of (LCP), solves the mixed boundary value problem:
$$  \Delta U = h         \ \mbox{ in } \Omega ,\ \ \ \ 
           U = f          \ \mbox{ at } \Gamma_1 ,\ \ \ \ 
       U_\nu = \bar{\vphi} \ \mbox{ at } \Gamma_2 . $$
Thus, we define the operator ${\cal L}_n: H^{\me}_{00}(\Gamma_2)' \ni 
\vphi \mapsto w_{|_{\Gamma_2}} \in H^{\me}(\Gamma_2)$, where $w$ solves
$$  \Delta w = h   \ \mbox{ in } \Omega ,\ \ \ \ 
           w = f    \ \mbox{ at } \Gamma_1 ,\ \ \ \ 
       w_\nu = \vphi \ \mbox{ at } \Gamma_2 . $$

\item[\it 2)] Notice that $U$ also solves the mixed boundary value problem:
$$  \Delta U = h                  \ \mbox{ in } \Omega ,\ \ \ \ 
       U_\nu = g                   \ \mbox{ at } \Gamma_1 ,\ \ \ \ 
           U = {\cal L}_n(\bar{\vphi}) \ \mbox{ at } \Gamma_2 . $$
This motivates the definition of the operator ${\cal L}_d: 
H^{\me}(\Gamma_2) \ni \psi \mapsto v_{\nu|_{\Gamma_2}} \in 
H^{\me}_{00}(\Gamma_2)'$, where $v$ solves
$$  \Delta v = h   \ \mbox{ in } \Omega ,\ \ \ \ 
       v_\nu = g    \ \mbox{ at } \Gamma_1 ,\ \ \ \ 
           v = \psi \ \mbox{ at } \Gamma_2 . $$

\item[\it 3)] Finally, we define the operator
\begin{equation} \label{gl:Tl-def}
{\cal T}: H^{\me}_{00}(\Gamma_2)' \, \ni \, \vphi \, \longmapsto \,
    {\cal L}_d({\cal L}_n(\vphi)) \, \in \, H^{\me}_{00}(\Gamma_2)'
\end{equation}
and observe that ${\cal T} \bar{\vphi} = \bar{\vphi}$. Reciprocally, if 
$\vphi$ is a fixed point of $\cal T$, it follows from an uniqueness 
argument for problem (LCP) that $\vphi = \bar{\vphi}$ (see, e.g., proof 
of Proposition~\ref{prop:equiv-cp-fpgl}).
\end{itemize}

\noindent {\bf Second approach:} Nonlinear mixed problems
\medskip

\noindent  In this approach we consider mixed boundary value problems 
similar to those introduced previously. However, no transformation 
is performed. We start by defining the operators
\begin{equation}
L_n: H^{\me}_{00}(\Gamma_2)' \to H^{\me}(\Gamma_2)\, ,\ \ 
L_d: H^{\me}(\Gamma_2) \to H^{\me}_{00}(\Gamma_2)'
\end{equation}
by $L_n(\vphi) := w_{|_{\Gamma_2}}$ and  $L_d(\psi) := q(v) 
v_{\nu|_{\Gamma_2}}$, where the $H^1(\Omega)$--functions $w$, $v$ solve 
the nonlinear mixed boundary value problems
\medskip

\centerline{$
         P w = h     \ \mbox{ in } \Omega  \, ,\ \ \ \ 
           w = f     \ \mbox{ at } \Gamma_1\, ,\ \ \ \ 
  q(w) w_\nu = \vphi \ \mbox{ at } \Gamma_2$}

\noindent  and

\centerline{$
         P v = h    \ \mbox{ in } \Omega  \, ,\ \ \ \ 
  q(v) v_\nu = g    \ \mbox{ at } \Gamma_1\, ,\ \ \ \ 
           v = \psi \ \mbox{ at } \Gamma_2$}
\medskip

\noindent  respectively. Now, defining the operator 
\begin{equation} \label{gl:T_def}
T: H^{\me}_{00}(\Gamma_2)' \, \ni \, \vphi \, \longmapsto \,
   L_d(L_n(\vphi)) \, \in \, H^{\me}_{00}(\Gamma_2)'
\end{equation}
and observing that $L_n(\bar{\vphi}) = u_{|_{\Gamma_2}}$ and 
$L_d(u_{|_{\Gamma_2}}) = \bar{\vphi}$, we obtain the desired 
characterization of $\bar{\vphi}$ as a solution of the fixed point 
equation for the operator $T$. In the next proposition we formalize 
this result as well as it's converse.

\begin{prop} \label{prop:equiv-cp-fpgl}
If problem (CP) admits a solution (say $u$), then $\bar{\vphi} := q(u) 
u_{\nu|_{\Gamma_2}}$ is a fixed point of the operator $T$ defined in 
\eqref{gl:T_def}. Conversely, if $\bar{\vphi}$ is a fixed point of $T$, 
the Cauchy problem (CP) is solvable and it's solution (say $u$) satisfies 
$q(u) u_{\nu|_{\Gamma_2}} = \bar{\vphi}$.
\end{prop}

\begin{proof}
The first part follows directly from the definition of $T$. Now let 
$\bar{\vphi}$ be a fixed point of $T$. The $H^1$--functions $w$ and $v$ 
satisfy
$$ w_{|_{\Gamma_2}} \, = \, v_{|_{\Gamma_2}} \ \ \mbox{ and } \ \ 
   q(w) w_{\nu|_{\Gamma_2}} \, = \, q(v) v_{\nu|_{\Gamma_2}} . $$
Argumenting with the uniqueness result in Theorem~\ref{satz:H1-eindeut-nlin} 
we conclude that $w \equiv v$. Thus, the function $w$ satisfies
\begin{equation}
         P w = h \ \mbox{ in } \Omega   ,\ \ \ \ 
           w = f \ \mbox{ at } \Gamma_1 ,\ \ \ \ 
  q(w) w_\nu = g \ \mbox{ at } \Gamma_1 ,
\end{equation}
i.e. $w$ is a solution of problem (CP). Since we have uniqueness of 
solutions for (CP) (again by Theorem~\ref{satz:H1-eindeut-nlin}), it 
follows $q(u) u_{\nu|_{\Gamma_2}} = q(w) w_{\nu|_{\Gamma_2}} = \bar{\vphi}$.
\end{proof}
\medskip

Using Theorems~\ref{satz:H1-eindeut-nlin} and \ref{prop:equiv-cp-fpgl} 
we can reinterpret the uniqueness result in Section~\ref{ssec:el-gem-prbl-nl} 
in terms of the fixed point equation $T(\vphi) = \vphi$:

\begin{corollary}
Given a pair of consistent Cauchy data $(f,g)$, the nonlinear operator 
$T$ defined in \eqref{gl:T_def} has exactly one fixed point in the 
Hilbert space $H^{\me}_{00}(\Gamma_2)'$.
\end{corollary}

\begin{remark}
In the first approach $\cal T$ is an affine operator, since ${\cal L}_n$ 
and ${\cal L}_d$ are both affine as well. The affine part of $\cal T$ 
depends only on the Cauchy data $(f,g)$ and on the right hand side $h$. 
Denoting this affine term by $z_{f,g} \in H^{\me}_{00}(\Gamma_2)'$ and 
the linear part of $\cal T$ by ${\cal T}_l$, it follows that 
$\bar{\vphi}$ is a solution of the linear Cauchy problem if and only 
if it is a solution of
\begin{equation}
(I - {\cal T}_l) \, \vphi \ = \ z_{f,g} .
\end{equation}
\end{remark}
%
%
%
\section{The Mann iteration} \label{sec:mann}

In this section we present a brief overview of the iterative method 
introduced by W.\;Mann in 1953 (see \cite{Ma}). Given a Banach space $X$ and 
$E \subset X$, Mann considered the problem of approximating the solution 
of the fixed point equation for a continuous operator $T: E \to E$.

In order to avoid the problem of existence of fixed points, Mann assumed the 
subset $E$ to be convex and compact (the existence question is than promptly 
answered by the Schauder fixed point theorem; see \cite{Sc}). Strongly 
influenced by the works of Ces\`aro and Topelitz, who used mean value methods 
in the summation of divergent series, Mann proposed a mean value iterative 
method based on the Picard iteration ($x_{k+1} := T(x_k)$), which we shall 
present next.

Let $A$ be the (infinite) lower triangular matrix
$$ A \, = \, \left( \begin{array}{cccccc}
     1      & 0      & 0      & \cdots & 0      & \cdots \\
     a_{21} & a_{22} & 0      & \cdots & 0      & \cdots \\
     a_{31} & a_{32} & a_{33} & \cdots & 0      & \cdots \\
     \vdots & \vdots & \vdots & \ddots & \vdots & \vdots
   \end{array} \right) , $$
with coefficients $a_{ij}$ satisfying 
\begin{itemize}
\item[\it i)]   $a_{ij} \, \ge \, 0\, ,\ i,j = 1, 2, \ldots$;
\item[\it ii)]  $a_{ij} \, = \, 0\, ,\ j > i$;
\item[\it iii)] $\T\sum\limits_{j=1}^i a_{ij} \, = \, 1\, , \ 
                i = 1, 2, \ldots$.
\end{itemize}
Starting with an arbitrary element $x_1 \in E$, the {\em Mann iteration} is 
defined by
\begin{tt}
\begin{itemize} \itemsep0.1ex
\item[1.] Choose $x_1 \in E$;
\item[2.] For $k = 1, 2, \ldots$ do \\[0.5ex]
     \mbox{\ \ } $v_k := \sum_{j=1}^k \, a_{kj} \, x_j$; \\[1ex]
     \mbox{\ \ } $x_{k+1} := T(v_k)$;
\end{itemize}
\end{tt}
and is briefly denoted by $M(x_1, A, T)$. Notice that with the particular 
choice $A = I$, this method corresponds to the usual Picard iterative 
process. Next we state the main theorem in \cite{Ma}.

\begin{lemma} \label{lemma:mann}

Let $X$ be a Banach space, $E \subset X$ a convex compact subset, $T: E \to E$ 
continuous. Further, let $\{ x_k \}$, $\{ v_k \}$ be the sequences generated 
by the iteration $M(x_1, A, T)$. If either of the above sequences converges, 
then the other also converges to the same point, and their common limit is a 
fixed point of $T$.
\end{lemma}

In that paper, the case where neither of the sequences $\{ x_k \}$, 
$\{ v_k \}$ converges is also considered. Under additional requirements 
on the coefficients $a_{ij}$, a relation between the sets of limit 
points of $\{ x_k \}$ and $\{ v_k \}$ is proven.

Many authors considered the Mann iteration in other frameworks. In the 
sequel we present some of the main results related, which will be useful 
for our further discussion.

In \cite{Op}, Z.\;Opial considers the very special case where $X$ is a 
Hilbert space, $E \subset X$ a closed convex subset, $T: E \to E$ a 
nonexpansive regular asymptotic application with non empty fixed point set, 
$A = I$. In the main theorem it is proven that $\{ x_k \}$ converges weakly 
to some fixed point of $T$ (see \cite{EnLe} for a generalization).

In \cite{Do}, W.\;Dotson extends the proof in \cite{Ma} to the case in 
which $X$ is a locally convex Hausdorff vector space and $E \subset X$ is a 
convex closed subset. This is achieved by using the regularity of the matrix 
$A$ and some properties of the continuous semi-norms which generate the 
topology of $E$.

In \cite{SeDo}, H.\;Senter and W.\;Dotson assume $X$ to be an uniformly convex 
Banach space, $E \subset X$ a closed bounded convex subset, $T: E \to E$ 
nonexpansive. Under special assumptions on the fixed point set of $T$, they 
prove that $\{ x_k \}$ converge to some fixed point of $T$.

In \cite{En}, H.W.\;Engl assumes $X$ to be an Opial normed space, 
$E \subset X$ weakly compact, $T: E \to X$ nonexpansive, $\sum a_{kk}$ 
divergent. In the main theorem it is proven that if $\{ x_k \}$ is well 
defined for some $x_1$ (i.e. $x_k \in E$, $\forall k$) then a (unique) 
fixed point $\bar{x}$ exists and $x_k$ converges weakly to $\bar{x}$. Further, 
$\{\bar{x}\}$ is characterized as the Chebyshef-center of the set $\{ x_k \}$.

In \cite{Gr1}, C.\;Groetsch considers a variant of the Mann iteration. The 
Matrix $A$ is assumed to be {\em segmenting}, i.e. additionally to properties 
\,{\it i)}, \,{\it ii)} and \,{\it iii)}, the coefficients $a_{ij}$ have also 
to satisfy
\begin{itemize}
\item[\it iv)] $a_{i+1,j} \, = \, (1 - a_{i+1,i+1}) \, a_{ij}$, \ $j \le i$.
\end{itemize}
One can easily check that, under assumptions {\it i)}, \dots, {\it iv)}, 
\,$v_{k+1}$ can be written as the convex linear combination
\begin{equation} \label{eq:segment}
v_{k+1} \ = \ (1 - d_k) v_k \, + \, d_k T(v_k) ,
\end{equation}
where $d_k := a_{k+1,k+1}$. In other words, $v_{k+1}$ lies on the line 
segment joining $v_k$ and $x_{k+1} = T(v_k)$, what justifies the 
denomination of {\em segmenting matrix}. Notice that the choice of 
the diagonal elements $d_k$ determines completely the matrix $A$. 
Next we enunciate the main theorem in \cite{Gr1}:

\begin{lemma} \label{lemma:groe}

Let $X$ be an {\em uniformly convex} Banach space, $E \subset X$ a convex 
subset, $T: E \to E$ a nonexpansive operator with at least one fixed point in 
$E$. If $\sum_{k=1}^\infty \, d_k (1-d_k)$ diverges, then the sequence 
$\{ (I-T) v_k \}$ converges strongly to zero, for every $x_1 \in E$.
\end{lemma}

In order to prove strong convergence of the sequence $\{ x_k \}$, one needs 
stronger assumption on both the set $E$ and the operator $T$ (e.g. $E$ is 
also closed and $T(E)$ is relatively compact in $X$).

Notice that Lemma~\ref{lemma:groe} gives on $\{ x_k \}$ a condition analogous 
to the {\em asymptotic regularity}, which is used in \cite{Op} (this condition 
is also used in \cite{Le} and \cite{EnLe} for the analysis of linear Cauchy 
problems).
%
%
%
\section{Iterative methods for nonlinear Cauchy problems} \label{sec:it_cp}

\subsection{Analysis of the first approach} \label{ssec:it_cp_1}

The first iterative method proposed in this paper corresponds to the Mann 
iteration applied to the fixed point equation ${\cal T} \vphi = \vphi$
with ${\cal T}$ defined in~\eqref{gl:Tl-def}. 
Initially we address the question of convergence for exact data. Given a 
pair of consistent Cauchy data $(f,g)$ we have the following result:

\begin{theorem}

Let $\cal T$ be the operator defined in~\eqref{gl:Tl-def} and $A$ a 
segmenting matrix with $\sum_{k=1}^{\infty} d_k (1-d_k) = \infty$. For 
every $\vphi_1 \in H^{\me}_{00}(\Gamma_2)'$ the iteration $(\vphi_1, 
A, {\cal T})$ generates sequences $\vphi_k$ and $\phi_k$, which converge 
strongly to $\bar{\vphi}$, the uniquely determined fixed point of $\cal T$.
\end{theorem}

\begin{sketch}
Existence and uniqueness of the fixed point $\bar{\vphi}$ follow from 
the assumption that $(f,g)$ are consistent data. Since $(I - {\cal T}_l) 
(\phi_k - \bar{\vphi})  = (I - {\cal T}) \phi_k$ and Ker$(I-{\cal T}_l) = 
\{ 0 \}$ (see \cite{Le}), it is enough to prove that $\lim_k (I - {\cal T}) 
\phi_k = 0$.

This however follows from \cite[Theorem~6]{EnLe}, which is the analog 
of Lemma~\ref{lemma:groe} for affine operators, with nonexpansive linear 
part, defined on Hilbert spaces.
\end{sketch}

Now let us consider the case of inexact data. We assume we are given noisy 
Cauchy data $(f_\eps,g_\eps)$, or alternatively $z_\eps$, such that
$$ \| z_{f,g} - z_\eps \| \leq \eps , $$
where $\eps>0$ is the noise level. This assumption is especially interesting 
in the case where $(f_\eps,g_\eps)$ correspond to measured data. In this 
case we consider the iteration residual and use the generalized {\em 
discrepancy principle} to provide a stopping rule for the algorithm, i.e 
the iteration is stopped at the step $k_\eps$ such that
$$  k_\eps := \min \{ k \in \N ;\ \| z_\eps - (I-{\cal T}_l) 
                                     \vphi_k \| \le \mu\eps \} $$
for some $\mu > 1$ fixed. If we do not make any further (regularity) 
assumption on the solution $\bar{\vphi}$, we cannot prove convergence 
rates for the iterates $\| \vphi_k - \bar{\vphi} \|$. However, it is 
possible to obtain rates of convergence for the residuals, as follows

\begin{prop}

If $\mu >1$ is fixed, the stopping rule defined by the {\em discrepancy 
principle} satisfies $k_\eps = O(\eps^{-2})$.
\end{prop}

The proof of this result is similar to the one known for the {\em Landweber 
iteration} (see, e.g. \cite[Section~6.1]{EHN}).

If appropriate regularity assumptions are made on the fixed point 
$\bar{\vphi}$, it is possible to obtain convergence rates also for 
the approximate solutions. This additional assumptions are stated 
here in the form of {\em source conditions} (see, e.g., \cite{EHN} or 
\cite{EnSc}). Since linear Cauchy problems are exponentially ill-posed 
(see, e.g., \cite{Le}), the source condition take the form
\begin{equation} \label{gl:source-cond}
\bar{\vphi} - \vphi_1 = f(I - {\cal T}_l) \psi ,
\end{equation}
where $\psi$ is some function in $H^{\me}_{00}(\Gamma_2)'$ and $f$ is 
defined by
$$  f(\lbd) := \left\{ \begin{array}{cl}
                 (\ln(e/\lbd))^{-p}, & \lbd > 0 \\
                                  0, & \lbd = 0
               \end{array} \right. $$
with $p>0$ fixed. This choice of $f$ corresponds to the so-called source 
conditions of logarithmic-type. Under these assumptions we can prove the 
following rates:

\begin{prop}

Let $k_\eps$ be the stopping rule determined by the discrepancy principle 
with $\mu > 2$. If the fixed point $\bar{\vphi}$ satisfies the source 
condition \eqref{gl:source-cond}, then we have

{\it i)} \ $k_\eps = O(\eps^{-1} (-\ln \sqrt{\eps})^{-p} )$;

{\it ii)} \ $\| \bar{\vphi} - \vphi_k \| = O( (-\ln \sqrt{\eps})^{-p})$.
\end{prop}

This result is a consequence of \cite[Theorem~15]{EnLe} (see also 
\cite{EnSc} for the non linear Landweber iteration). The interest in 
the source conditions of logarithmic-type for this fixed point equation 
is motivated by the fact that it can be interpreted in the sense of 
$H^s$ regularity. Indeed, it is shown in \cite{EnLe} that condition 
\eqref{gl:source-cond} is equivalent to assume, in a special case 
where the spectral decomposition of the operator ${\cal T}_l$ is known, 
that $\bar{\vphi} - \vphi_1$ is in the Sobolev space $H^p$.

\subsection{Analysis of the second approach} \label{ssec:it_cp_2}

Let $T$ be the operator defined in \eqref{gl:T_def}. We assume the 
Cauchy data $(f,g)$ to be consistent and denote by $\bar{\vphi}$ the 
fixed point of $T$. We start by defining the operator
\begin{equation} \label{gl:barT_def}
\overline{T} \vphi \ := \ 
\begin{cases}
  {T} \vphi &\!\!\! , \| {T} \vphi \| \le \|\vphi\| \\
  \frac{\|\vphi\|}{\| {T} \vphi \|} {T} \vphi
                 &\!\!\! , \| {T} \vphi \| \ge \|\vphi\|
\end{cases}
\end{equation}
The operator $\overline{T}$ is continuous, nonlinear and further satisfies 
$ \| \overline{T} \vphi \| \le \|\vphi\|$, for all $\vphi \in 
H^{\me}_{00}(\Gamma_2)'$. As one can easily check, every fixed point of 
$T$ is also a fixed point of $\overline{T}$. The reciprocal, however, is 
not true. What we can prove is the following:

\begin{prop} An element $\vphi \in H^{\me}_{00}(\Gamma_2)'$ is a fixed 
point of $T$ iff it is a fixed point of $\overline{T}$ and 
$\| {T} \vphi \| \le \| \vphi \|$.
\end{prop}

Next we introduce an iterative method (based on the Mann iteration) 
to approximate the fixed points of $\overline{T}$. Given a 
segmenting matrix $A$, consider the algorithm
\begin{tt}
\begin{itemize} \itemsep0.1ex
\item[1.] Choose $\vphi_1 \in H^{\me}_{00}(\Gamma_2)'$;
\item[2.] For $k = 1, 2, \ldots$ do \\[0.5ex]
      \mbox{\ \ } $\phi_k := \sum_{j=1}^k \, a_{kj} \, \vphi_j$; \\[1ex]
      \mbox{\ \ } $\vphi_{k+1} := \overline{T}(\phi_k)$;
\end{itemize}
\end{tt}
Notice that $\vphi_k, \phi_k \in H^{\me}_{00}(\Gamma_2)'$, $k = 1,2,\dots$. 
We represent this iterative process by $(\vphi_1, A, \overline{T})$. 
Obviously this iteration coincides with the Picard iteration if one chooses 
$A = I$.%
\footnote{In the linear case the choice $A = I$ corresponds to the Maz'ya 
iteration; see \cite{EnLe}.}

Now we discuss an auxiliary lemma concerning uniformly convex spaces.%
\footnote{See \cite{Ad} for the corresponding definitions.}

\begin{lemma} \label{lem:groe-aux}

Let $X$ be an uniformly convex linear space with modulus of convexity 
$\delta$. Further let $\eps > 0$, $d > 0$ and $\lbd \in [0,1]$ be given. 
If $\phi, \psi \in X$ are such that
$$ \| \psi \| \ \le \ \| \phi \| \ \le \ d \ \ \mbox{ and } \ \ 
   \| \phi - \psi \| \ \ge \ \eps , $$
then
$$ \| (1-\lbd) \phi + \lbd \psi \| \ \le \ \| \phi \| \,
   \big( 1 - 2\delta(d^{-1}\eps) \, \min\{ \lbd, 1-\lbd\} \big) . $$
\end{lemma}

\begin{proof} See \cite{Gr2}.
\end{proof}

Next we prove for our iterative method, a result analogous to the one 
stated Lemma~\ref{lemma:groe}. Notice that, different to that lemma, 
we do not assume nonexpansivity. Instead we use the weaker property 
$\| \overline{T} \vphi \| \le \|\vphi\|$ of $\overline{T}$.

\begin{prop} \label{satz:conv_iter}

Let $\overline{T}$ be the operator defined in \eqref{gl:barT_def} 
and $A$ a segmenting matrix such that $\sum_{k=1}^\infty d_k (1-d_k)$ 
diverges. Further let $\vphi_1 \in H^{\me}_{00}(\Gamma_2)'$ and 
$\{ \phi_k \}$ be the sequence generated by the iteration 
$(\vphi_1, A, \overline{T})$. Then, there exists a subsequence 
$\{ \phi_{k_j} \}$ such that $\{ (I-\overline{T}) \phi_{k_j} \}$ 
converges (strongly) to zero.
\end{prop}

\begin{proof}
Given $\vphi_1 \in H^{\me}_{00}(\Gamma_2)'$, let $\{ \vphi_k \}$ and 
$\{ \phi_k \}$ be the sequences generated by the iteration $(\vphi_1, A, 
\overline{T})$. Setting $\psi_k := \vphi_{k+1}$, for $k \ge 0$, it 
follows from the definition of $\overline{T}$ and from the segmenting 
property \eqref{eq:segment} that
\begin{equation} \label{gl:satz-conv-1}
\| \psi_k \| \ \le \ \| \phi_k \| \ \le \ \| \vphi_1 \|\, ,\ k=1,2,\dots
\end{equation}
and
\begin{equation} \label{gl:satz-conv-2}
\phi_{k+1} \ = \ (1 - d_k)\, \phi_k + d_k\, \psi_k\, ,\ k=1,2,\dots\, .
\end{equation}
Since $(I-\overline{T}) \phi_k = \phi_k - \psi_k$, $k \ge 1$, it 
is enough to prove that $0 \in$ cl$\{\phi_k - \psi_k\}$, where cl\,$\{M\}$ 
denotes the (strong) closure of the set $M$ in $H^{\me}_{00}(\Gamma_2)'$.

Now, suppose there exists $\eps > 0$ such that
\begin{equation} \label{gl:satz-conv-3}
\|\phi_k - \psi_k\| \ge \eps\, ,\ k=1,2,\dots\, .
\end{equation}
Thus, it follows from Lemma~\ref{lem:groe-aux}, together with 
\eqref{gl:satz-conv-1} and \eqref{gl:satz-conv-2}
\begin{eqnarray*}
\| \phi_{k+1} \| &  =  & \| (1 - d_k)\, \phi_k + d_k\, \psi_k \| \\
                 & \le & \| \phi_k \|\, \big( 1 - 2\delta(\|\vphi_1\|^{-1}\eps)
                         \, \min\{ d_k, 1-d_k \} \big) .
\end{eqnarray*}
Repeating inductively the argumentation we obtain
\begin{equation} \label{gl:satz-conv-4}
 \| \phi_{k+1} \| \ \le \ \| \phi_1 \| \, \prod_{j=1}^k
   \big( 1 - 2\delta(\|\vphi_1\|^{-1}\eps) \, \min\{ d_j, 1-d_j \} \big) .
\end{equation}
However, since $d_j(1-d_j) \le \min\{ d_j, 1-d_j \}$, for $j \ge 1$, it 
follows from the assumption on the series $\sum d_j (1-d_j)$ that
$$ \sum_{j=1}^\infty \, \min\{ d_j, 1-d_j \} \ = \ \infty . $$
Consequently, the product on the right hand side of \eqref{gl:satz-conv-4} 
converges to zero and hence $\lim \| \phi_k \| = \lim \| \psi_k \| = 0$. 
However, this contradicts \eqref{gl:satz-conv-3}, completing the proof.
\end{proof}

Notice that the assumption $\sum_{k=1}^\infty d_k (1-d_k) = \infty$ in 
Theorem~\ref{satz:conv_iter} can be replaced by the requirement 
$\sum_{j=1}^\infty \min\{ d_j, 1-d_j \} = \infty$. The proof remains 
the same.

\begin{remark}

In \cite[Lemma~2]{Gr2}, C.\;Groetsch obtains an analogous result for an 
iterative process introduced by W.\;Kirk (see \cite{Ki}). In that paper, 
fixed points of a nonexpansive operator $T$ are approximated by a sequence 
of the type $x_{n+1} = S_n x_n$, where the operators $S_n$ are defined by
$$ S_n \ := \ \alpha_{n0}I + \alpha_{n1}T + \cdots + \alpha_{nk}T^k\, ,
   \ n = 1,2,\dots\, , $$
with $\alpha_{ij} \ge 0$, $\alpha_{n1} \ge \alpha > 0$, 
$\sum\limits_{j=1}^k \alpha_{nj} = 1$, $j=1,\dots,k$ and 
$\sum\limits_{n=1}^\infty \min\{ \alpha_{n0}, 1-\alpha_{n0} \} = \infty$.

One should notice that in \cite[Lemma~2]{Gr2} as well as in Lemma~%
\ref{lemma:groe}, one needs for the proof the fact that the operator 
$T$ is nonexpansive. However, we have only used the property 
$\| \overline{T} \vphi\| \le \| \vphi \|$, for all $\vphi$.
\end{remark}

\begin{remark}

If for a particular coefficient function $q$ the operator 
$\overline{T}$ happens to be nonexpansive, then one can argue 
as in \cite[Corollary~7]{EnLe} to prove that, under the same 
assumptions of Proposition~\ref{satz:conv_iter}, the sequence 
$\{ \phi_k \}$ generated by the iteration $(\vphi_1, A, \overline{T})$ 
converges (strongly) to a fixed point of $\overline{T}$ for 
every $\vphi_1 \in H^{\me}_{00}(\Gamma_2)'$.
\end{remark}
%
%
%
\section{Numerical experiments} \label{sec:numeric}

In the next paragraphs we present some numerical results, which correspont to 
the implementation of the iterative method proposed in Section~%
\ref{ssec:it_cp_2}. The first two examples concern consistent Cauchy problems 
(exact data) in a square domain. In the third example we consider a problem 
with noisy data.

The computation was performed on the Silicon Graphics SGI-machines (based 
on R12000 processors; 32-bit code) at the Spezialforschungsbereich F013. 
The elliptic mixed boundary value problems were solved using the NETLIB 
software package PLTMG (see \cite{Ba}).

\subsection{A problem with harmonic solution} \label{ssec:num-ex1}

Let $\Omega \subset \R^2$ be the open rectangle $(0,1) \times (0,1/2)$ and 
define the following subsets of $\partial\Omega$:
$$         \Gamma_1 := \{ (x,0) ; x \in (0,1) \}\, , \ \ \ \ \ \ 
           \Gamma_2 := \{ (x, 1/2) ; x \in (0,1) \}\, , $$
$$        \Gamma_3 := \{ (0,y) ; y \in (0,1/2) \}\, , \ \ \ \,
          \Gamma_4 := \{ (1,y) ; y \in (0,1/2) \} . $$
Let $q(t) = 1 + t^2$ and $\bar{u}: \bar{\Omega} \to \R$ be the harmonic 
function
$$ \bar{u}(x,y) \ := \ x^2 - y^2 + 5x + 2y - 3xy . $$
We consider the Cauchy problem
$$ \left\{ \begin{array}{rcl}
     -\nabla \cdot (q(u) \nabla u) & = & h \, , \mbox{ in } \Omega \\
     u           & = & f       \, , \mbox{ at } \Gamma_1 \\
     q(u)u_{\nu} & = & g       \, , \mbox{ at } \Gamma_1 \\
     u           & = & \bar{u} \, , \mbox{ at } \Gamma_3 \cup \Gamma_4
   \end{array} \right.  $$
where the Cauchy data
$$ f(x) \ = \ x^2 + 5x \, ,\ \ \ 
   g(x) \ = \ (1 + (x^2 + 5x)^2) \, (3x - 2) $$
is given at $\Gamma_1$ and the right hand side $h : \Omega \to \R$ is given by
$$ h(x,y) \ = \ - 2 \bar{u}(x,y) \, | \nabla \bar{u}(x,y) |^2 . $$
We aim to reconstruct the (Dirichlet) trace of $u$ at $\Gamma_2$. The problem 
was artificially constructed and one can easily check that the desired trace 
is given by
$$ \bar{\vphi} \ = \ \T x^2 + \frac{7}{2}x + \frac{3}{4} \ 
                (= \, \bar{u}_{|_{\Gamma_2}}) . $$

We used in the iteration the Ces\`aro matrix
$$ A \, = \, \left( \begin{array}{cccccc}
     1   & 0    & 0      & \cdots & 0      & \cdots \\
     1/2 & 1/2  & 0      & \cdots & 0      & \cdots \\
     1/3 & 1/3  & 1/3    & \cdots & 0      & \cdots \\
     \vdots     & \vdots & \vdots & \ddots & \vdots & \vdots
   \end{array} \right) . $$
As initial guess, $\vphi_1 \equiv 0$ was chosen (at the end points $x=0$ and 
$x=1$, $\vphi_1$ must take the values 3/4 and 21/4 respectively, in order to 
be compatible with the other boundary conditions). Each mixed boundary value 
problem was solved using a (multi-grid) finite element method, with linear 
elements and a uniform mesh with 1\,073 nodes (33 nodes on $\Gamma_2$).

\begin{figure}[t] \unitlength1cm
\begin{center}
\begin{picture}(14,8.5)
\put(0,4.5){\centerline{
\epsfxsize4cm\epsfysize4cm \epsfbox{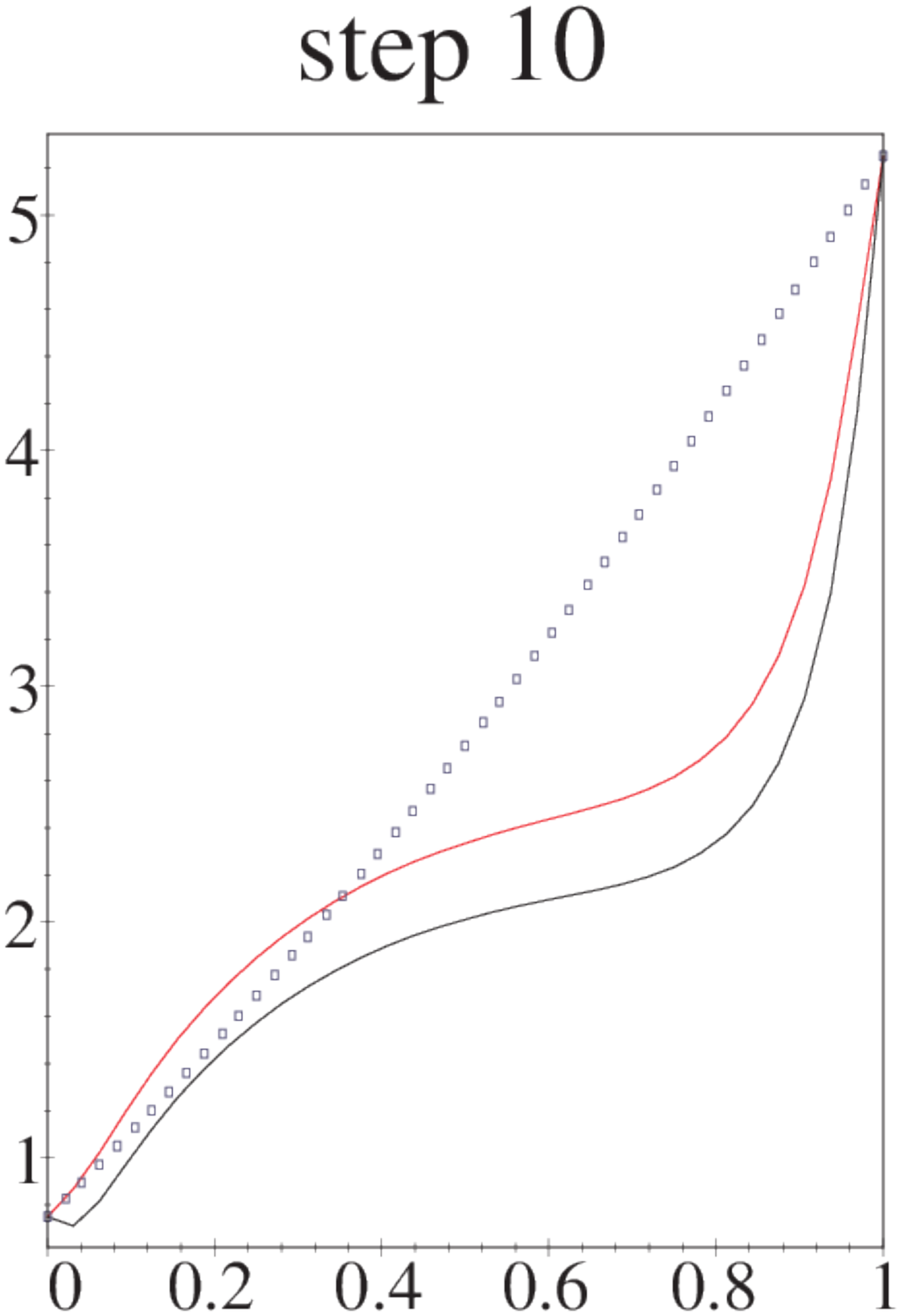} \hskip0.7cm
\epsfxsize4cm\epsfysize4cm \epsfbox{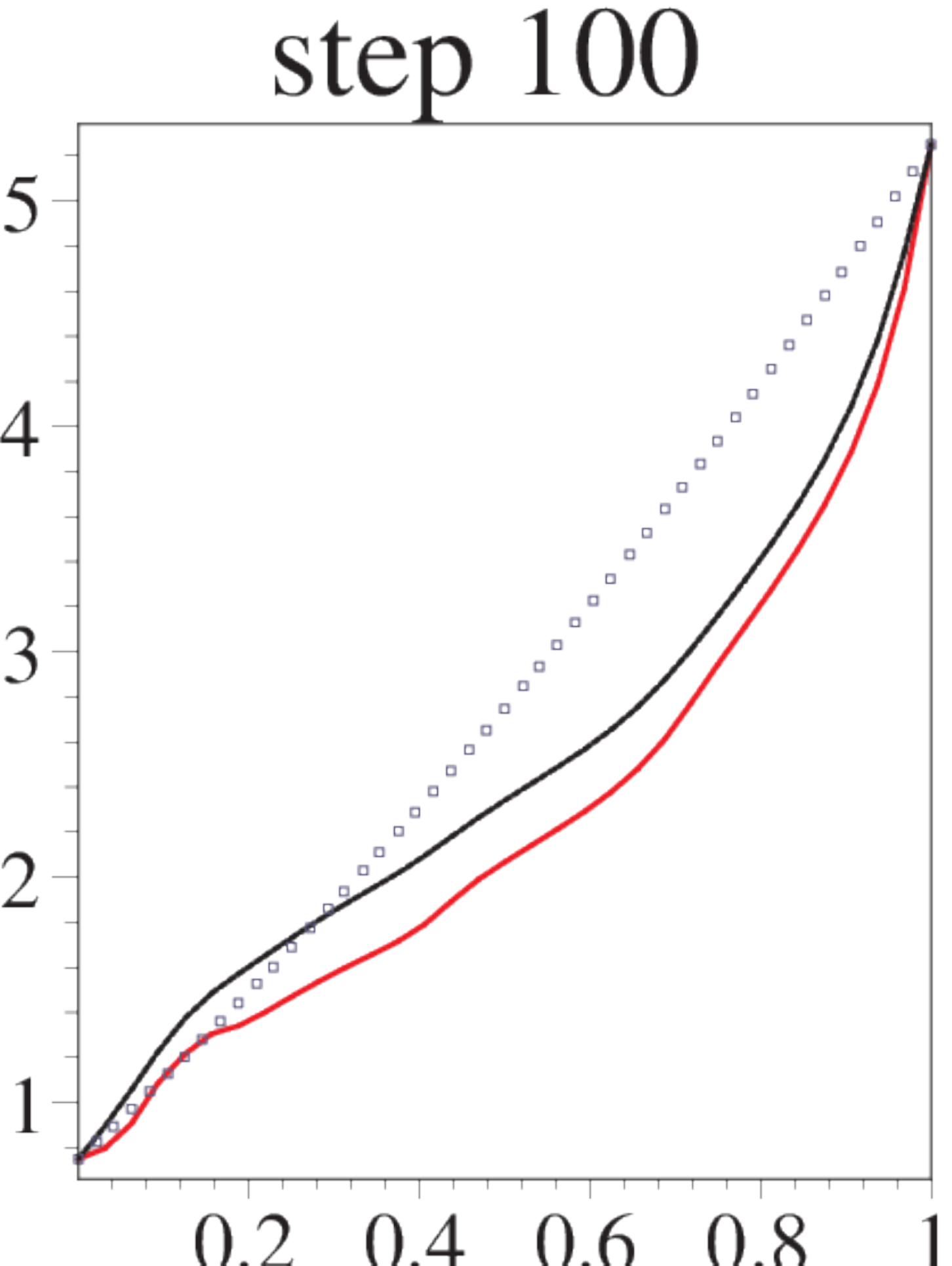} \hskip0.7cm
\epsfxsize4cm\epsfysize4cm \epsfbox{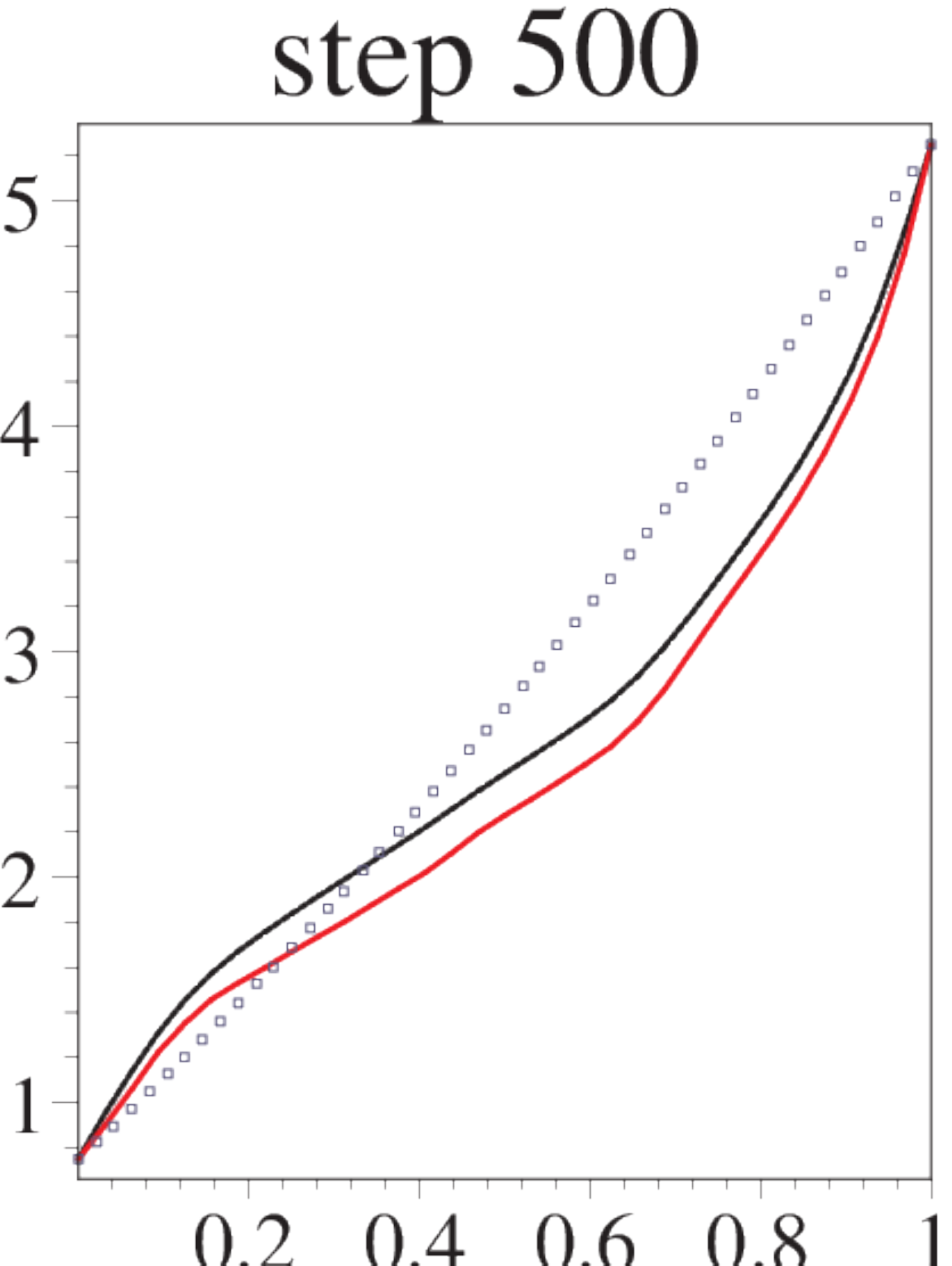} } }
\put(0,0){\centerline{
\epsfxsize4cm\epsfysize4cm \epsfbox{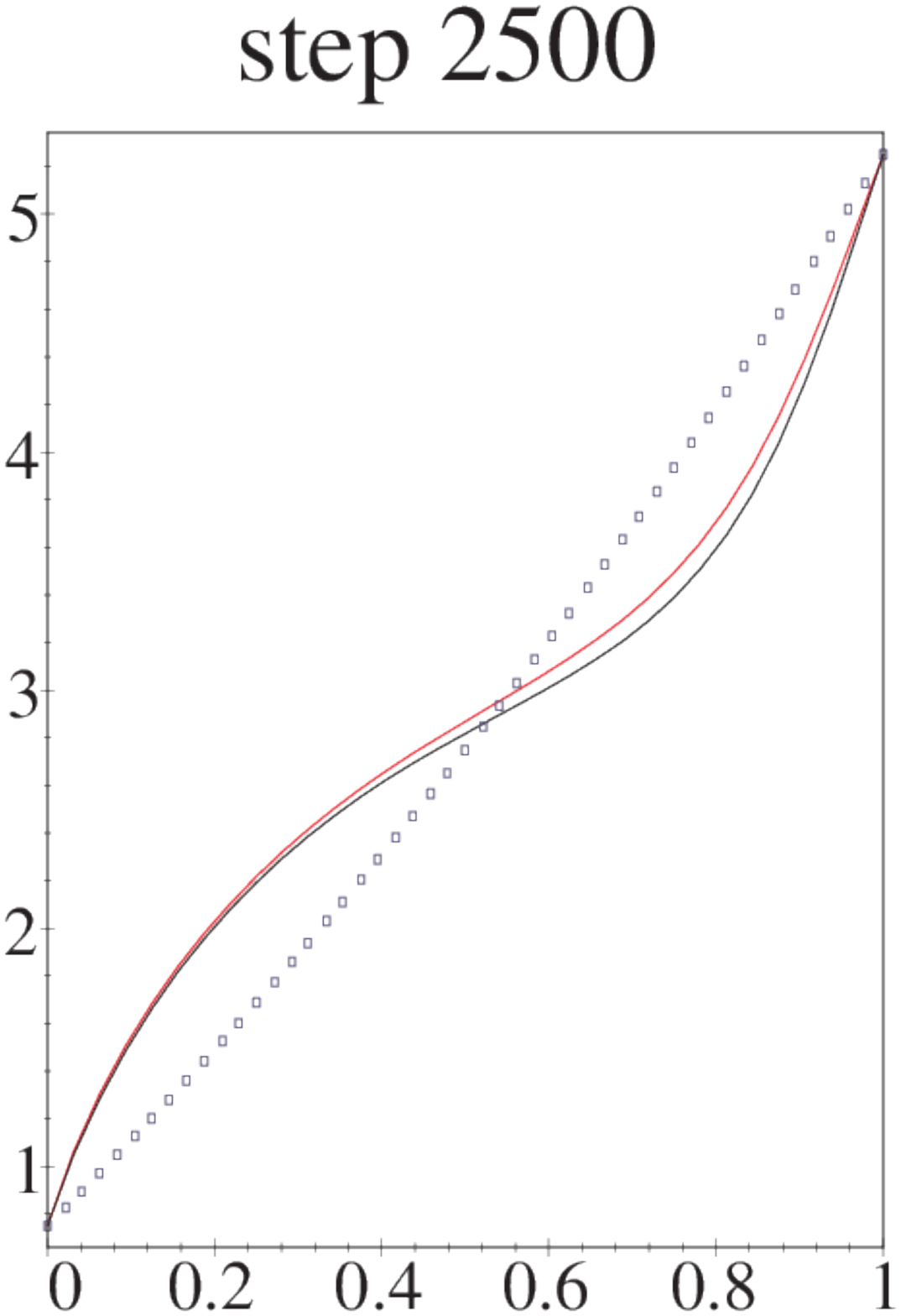} \hskip1cm
\epsfxsize4cm\epsfysize4cm \epsfbox{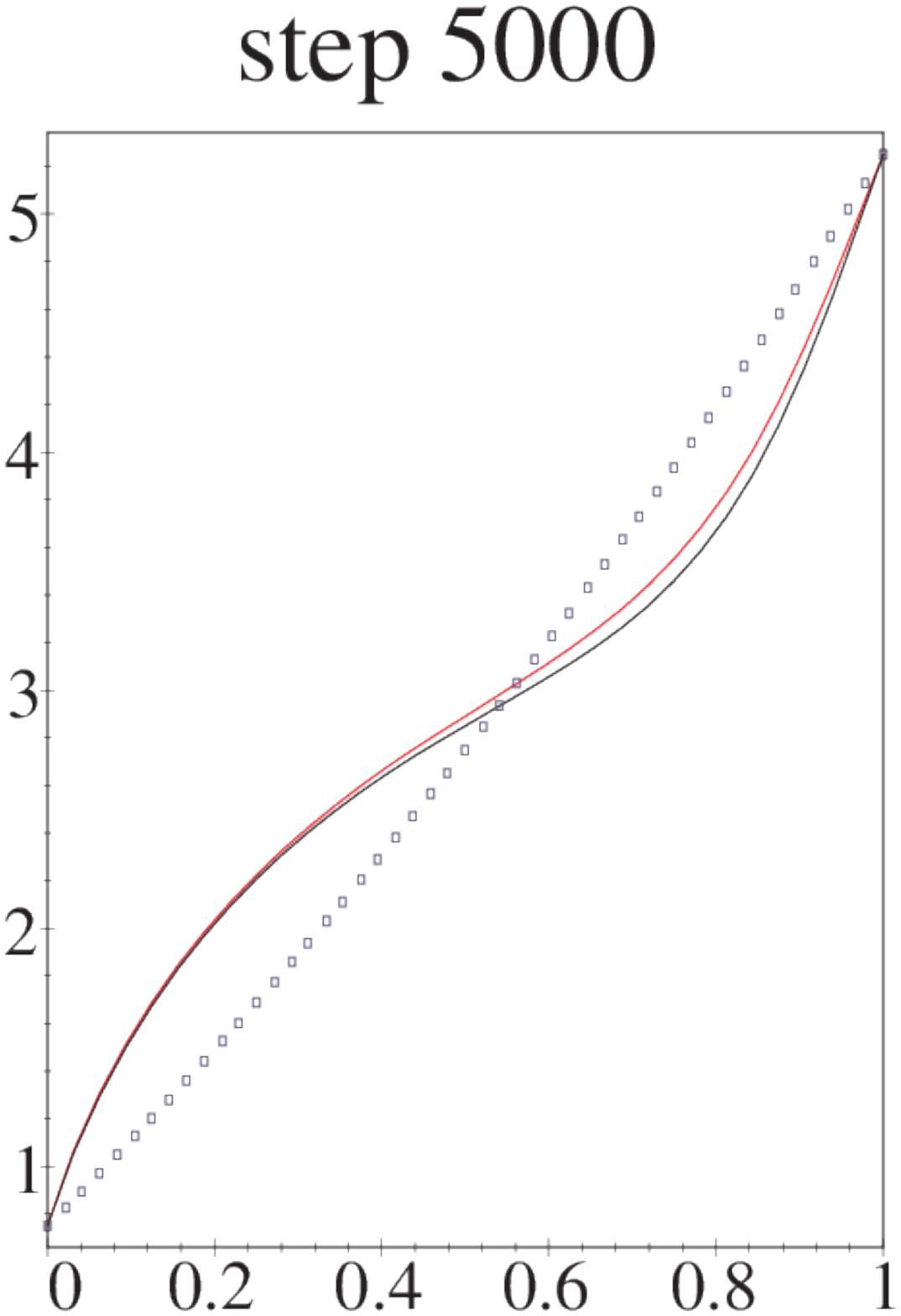} } }
\end{picture}
\end{center} \vskip-0.6cm
\caption{Iteration for a Cauchy problem with harmonic solution 
\label{fig:ex1-f1}}
\end{figure}

We used the stopping rule $\| \vphi_k - \vphi_{k-1} \|_{L^2} \le 10^{-2}$ 
(the same used in \cite{EnLe} for the linear case). In Figure~%
\ref{fig:ex1-f1} we present the results corresponding to the Mann iteration 
for the operator $\bar{T}$. The dotted (blue) line represents the exact 
solution, the dashed (black) line represents the sequence $\vphi_k$ and the 
solid (red) line represents the sequence $\psi_k$, both generated by 
$(0,A,\bar{T})$.

>From the results in Figure~\ref{fig:ex1-f1}, one can conclude that the 
convergence rate decays very fast with the iteration. This can be in part 
explained by the linear convex combination used to compute $\psi_k$ in the 
Mann iteration (note that\, $\psi_{k+1} - \psi_k = \frac{1}{k+1} \vphi_{k+1} 
- \frac{1}{k(k+1)} \sum_{j=1}^k \vphi_k$). We consider the following 
alternative to improve the convergence rate:
\begin{itemize} \itemsep0.2ex
\item Choose $\eps' > 0$ larger then the precision ($\eps > 0$) to be achieved;
\item Use the stopping rule $\|\psi_k-\psi_{k-1}\|_{L^2} \le \eps'$;
\item Restart the iteration with $\vphi_1 = \psi_k$;
\item The iteration should be renewed restarted until $\| T\vphi_k - 
\vphi_k \|_{L^2} \le \eps$.
\end{itemize}
Notice that we restart the iteration every time the mean value $\psi_k$ stops 
changing significantly. The restart procedure should be repeated until 
$\vphi_k$ (or alternatively $\psi_k$) approximates the fixed point of 
$\overline{T}$ with the desired precision. In Figure~\ref{fig:ex1-f2} 
we present the results corresponding to this restart strategy (the meaning 
of the curves is the same as in Figure~\ref{fig:ex1-f1}). For comparison 
purposes, we restarted the iteration after every 50 steps. Thus, in order 
to compute the results in Figure~\ref{fig:ex1-f2}, we had to evaluate 100, 
200 and 300 iteration steps respectively.

One should notice that the result obtained after the third restart (middle  
picture in Figure~\ref{fig:ex1-f2}) required 200 iteration steps to be 
computed and already gives us a much better approximation to the actual 
solution than the one obtained after 5000 steps with the Mann method.

\begin{figure}[t] \unitlength1cm
\begin{center}
\begin{picture}(14,4)
\centerline{
\epsfxsize4cm\epsfysize4cm \epsfbox{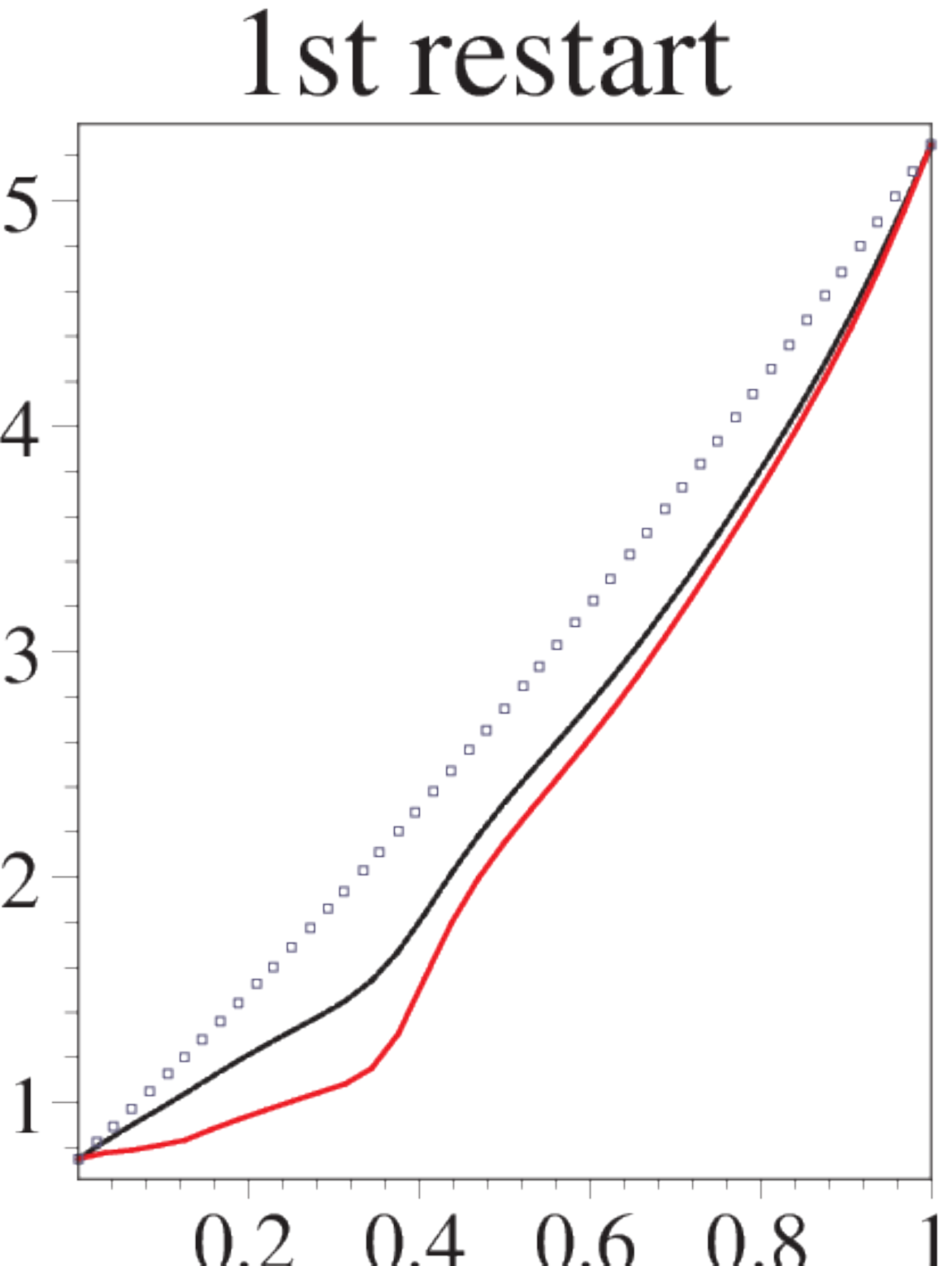} \hskip0.7cm
\epsfxsize4cm\epsfysize4cm \epsfbox{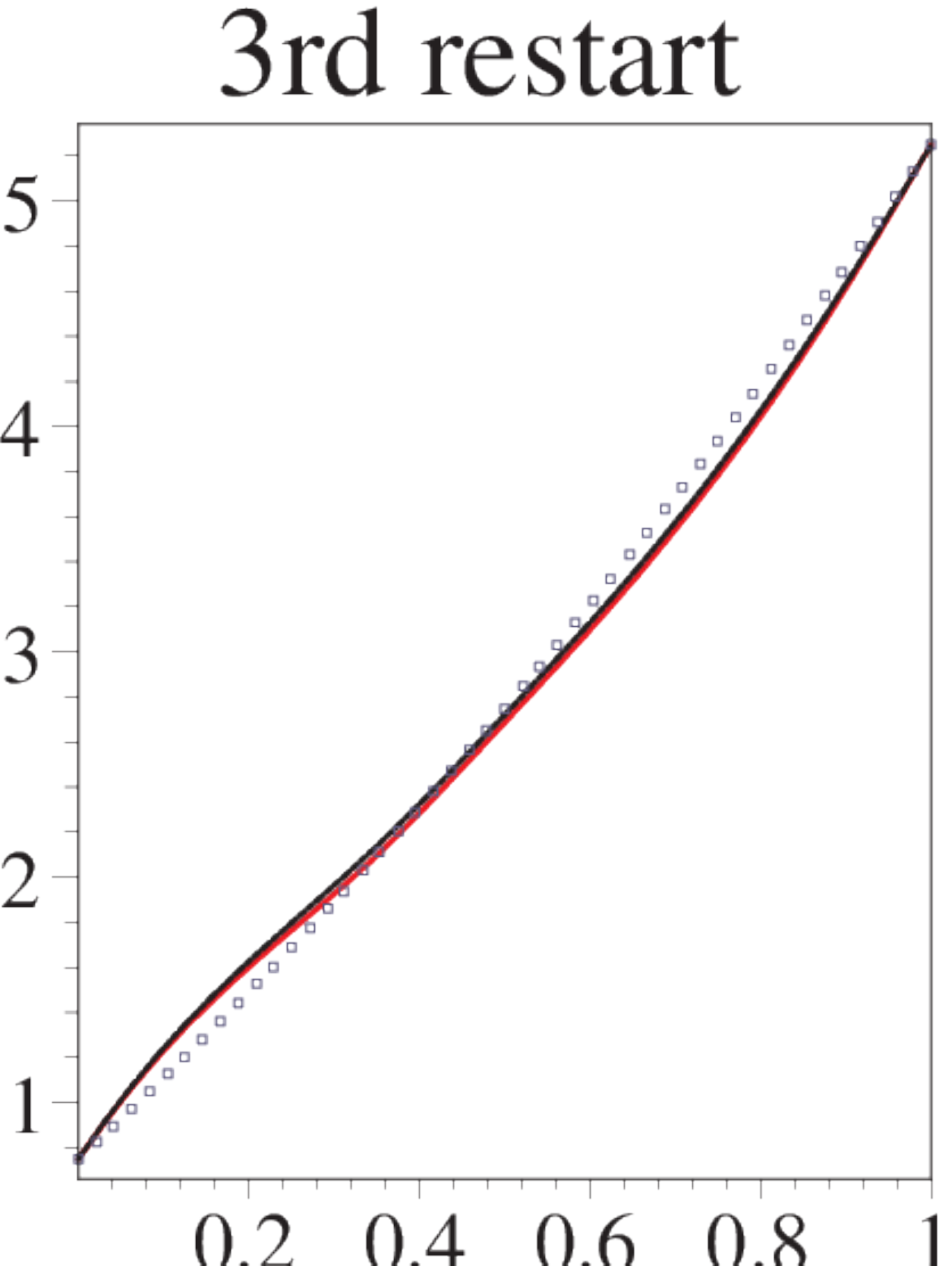} \hskip0.7cm
\epsfxsize4cm\epsfysize4cm \epsfbox{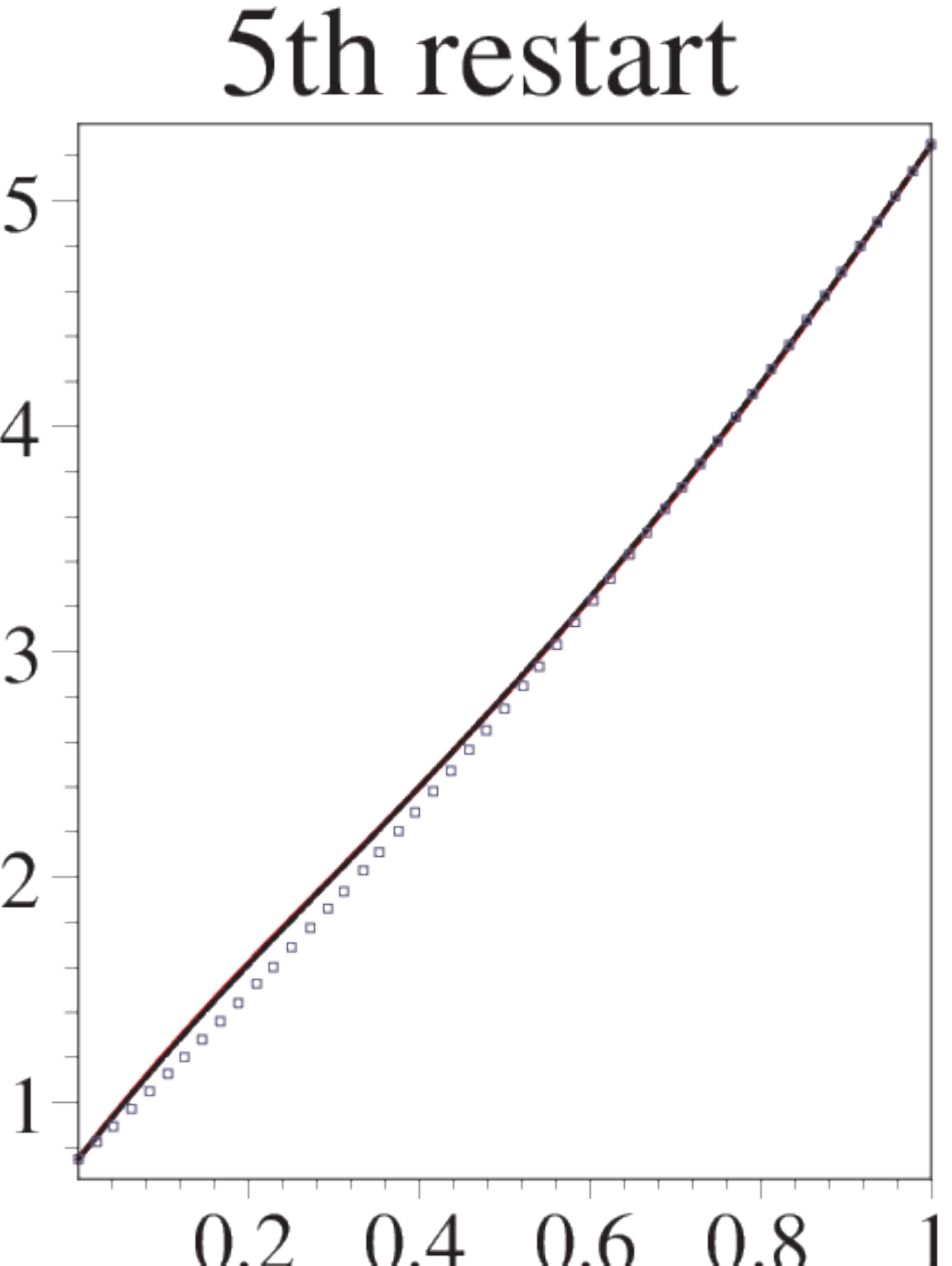} }
\end{picture}
\end{center} \vskip-0.6cm
\caption{Iteration with restart strategy (after every 50 steps) for a Cauchy 
problem with harmonic solution \label{fig:ex1-f2}}
\end{figure}

\subsection{A problem with non harmonic solution} \label{ssec:num-ex2}

Let $\Omega \subset \R^2$ and $\Gamma_i \subset \partial\Omega$, $i=1, \dots, 
4$ be defined as in the previous section. Let $q(t) = 2 + \sin t$. We consider 
the Cauchy problem
$$ \left\{ \begin{array}{rcl}
     -\nabla \cdot (q(u) \nabla u) & = & h \, , \mbox{ in } \Omega \\
     u           & = & f      \, , \mbox{ at } \Gamma_1 \\
     q(u)u_{\nu} & = & g      \, , \mbox{ at } \Gamma_1 \\
     u           & = & \bar{u}\, , \mbox{ at } \Gamma_3 \cup \Gamma_4
   \end{array} \right.  $$
with Cauchy data
$$ f(x) \ = \ \cos \pi x \, ,\ \ \ \T
   g(x) \ = \ - \big( 2 + \sin( \cos \pi x) \big) \, \cos(\pi x) $$
given at $\Gamma_1$ and right hand side $h : \Omega \to \R$ given by
$$ h(x,y) \ = \ (2 + \sin \bar{u}) (\pi^2-1) \bar{u} \, - \, \cos \bar{u}
              \, | \nabla \bar{u} |^2 , $$
where $\bar{u}: \bar{\Omega} \to \R$ is defined by
$$ \bar{u}(x,y) \ := \ \T \cos(\pi x) \, \exp(y) . $$
As in the previous example, the Cauchy problem was so constructed, such that 
$\bar{\vphi}$ is known. Indeed we have $\bar{\vphi} = \bar{u}_{|_{\Gamma_2}}$.

\begin{figure}[t] \unitlength1cm
\begin{center}
\begin{picture}(14,8.5)
\put(0,4.5){\centerline{
\epsfxsize4cm\epsfysize4cm \epsfbox{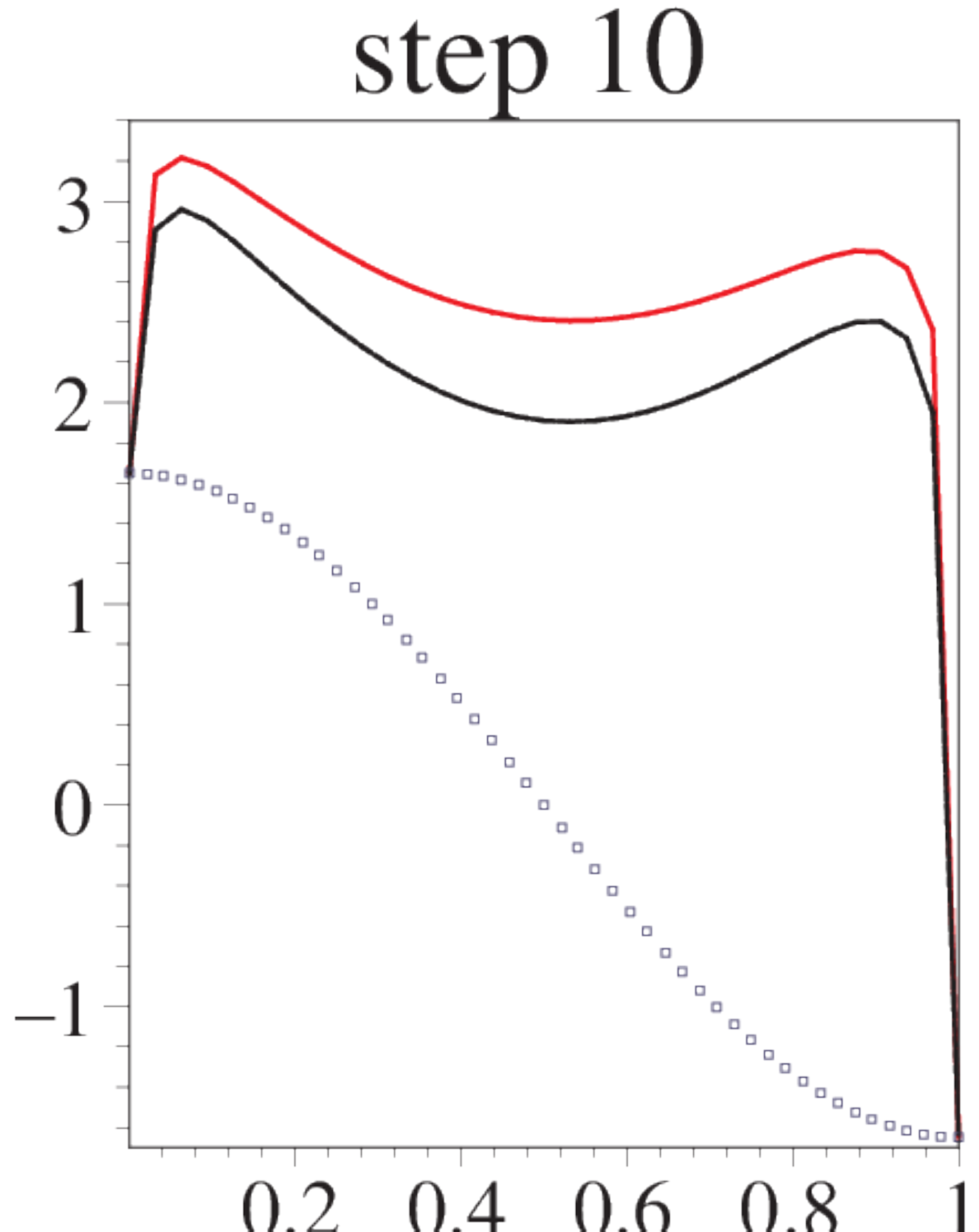} \hskip0.7cm
\epsfxsize4cm\epsfysize4cm \epsfbox{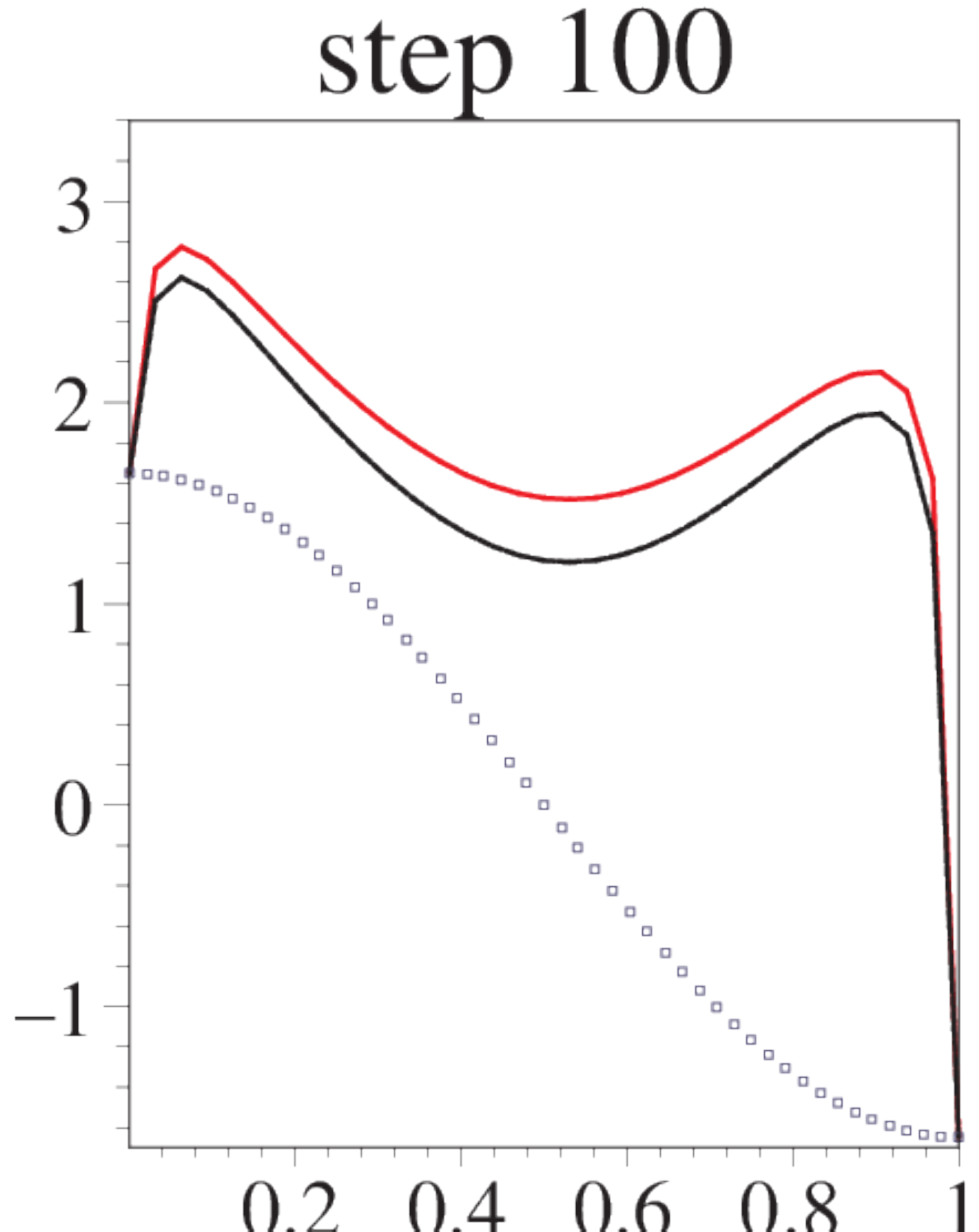} \hskip0.7cm
\epsfxsize4cm\epsfysize4cm \epsfbox{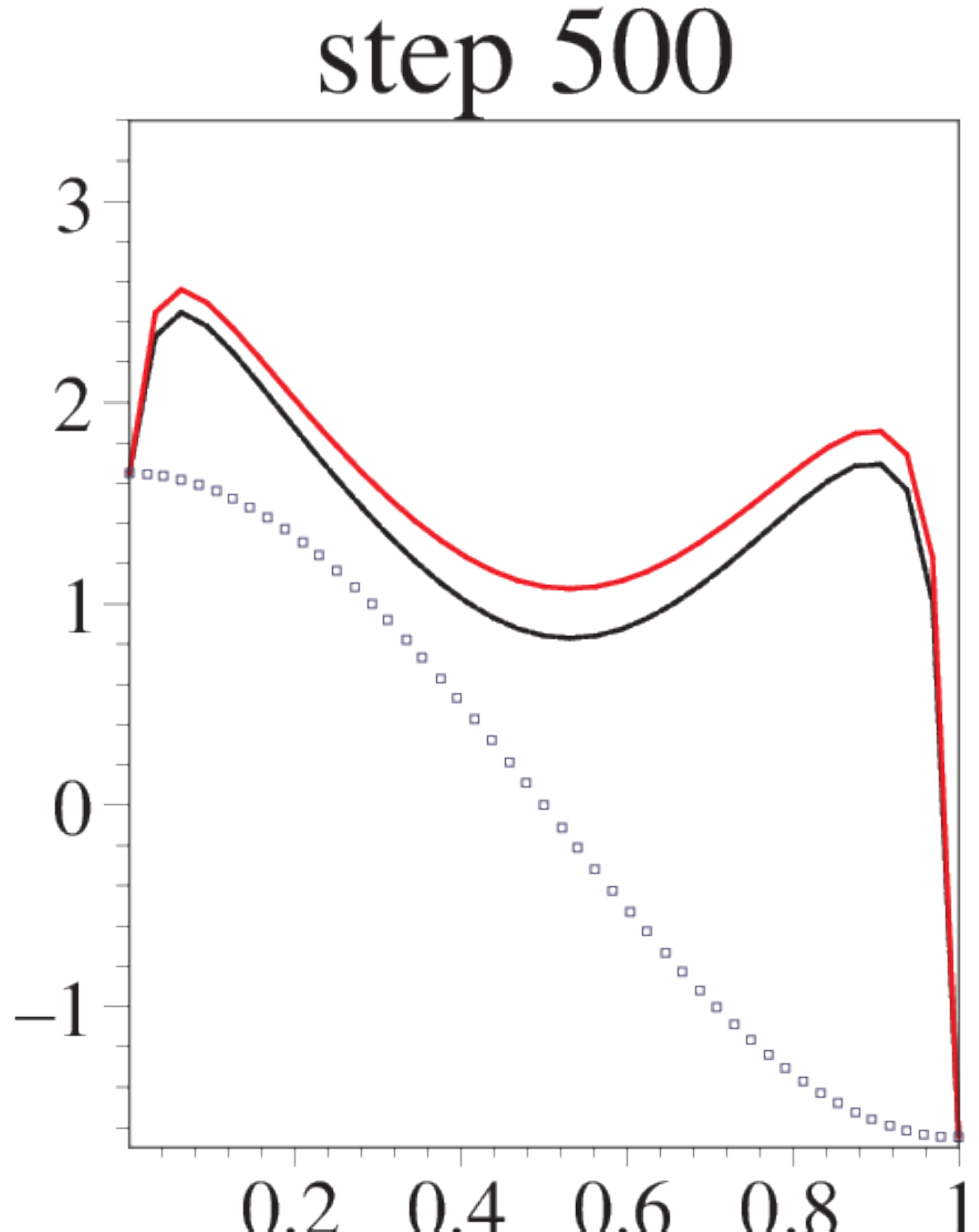} } }
\put(0,0){\centerline{
\epsfxsize4cm\epsfysize4cm \epsfbox{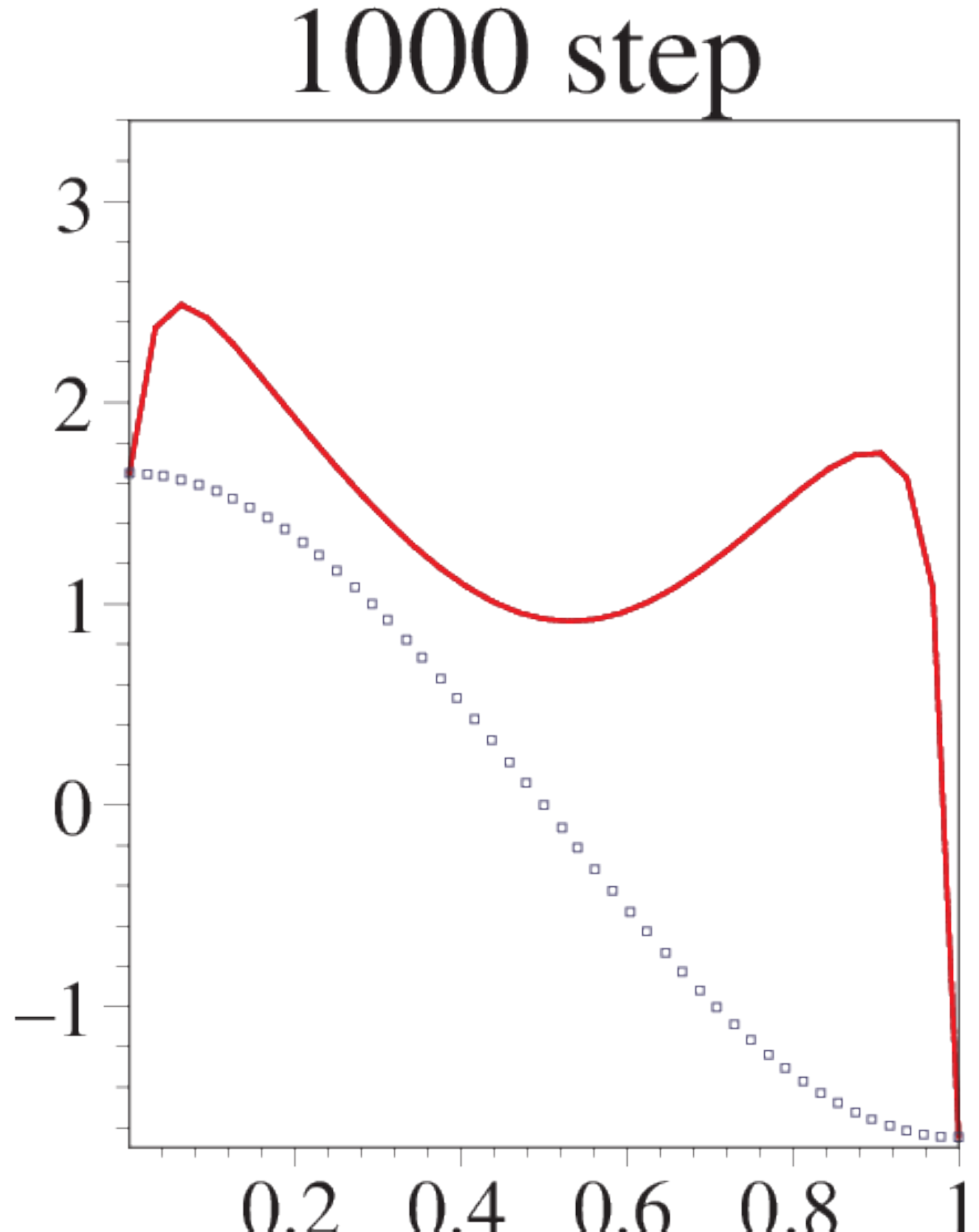} \hskip1cm
\epsfxsize4cm\epsfysize4cm \epsfbox{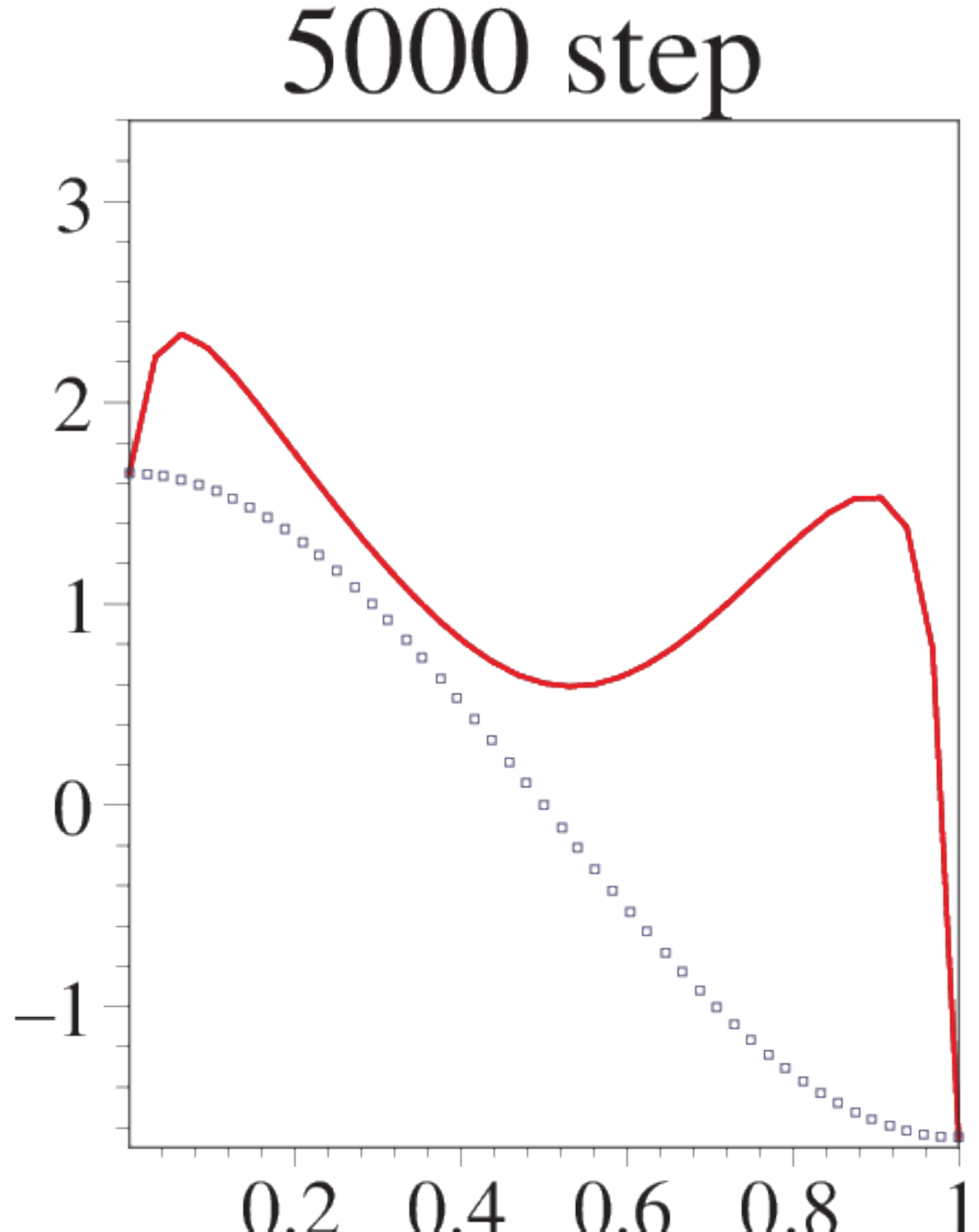} } }
\end{picture}
\end{center} \vskip-0.6cm
\caption{Iteration for a Cauchy problem with non harmonic solution 
\label{fig:ex2-f1}}
\end{figure}

For the numerical computations we used the same segmenting matrix $A$ and the 
same stopping rule as in the previous example. As initial guess, $\vphi_1 
\equiv 4$ was chosen (at the end points $x=0$ and $x=1$ we must choose, for 
compatibility reasons, $\vphi_1(0) = \exp(0.5)$, $\vphi_1(1) = -\exp(0.5)$). 
Each mixed boundary value problem was solved using a (multi-grid) finite 
element method, with linear elements and a uniform mesh with 16\,577 nodes 
(129 nodes on $\Gamma_2$).

In Figure~\ref{fig:ex2-f1} we present the results corresponding to the Mann 
iteration for the operator $\bar{T}$ (the meaning of the curves is the same 
as in Figures~\ref{fig:ex1-f1} and \ref{fig:ex1-f2}).

\begin{figure}[t] \unitlength1cm
\begin{center}
\begin{picture}(14,4)
\centerline{
\epsfxsize3.3cm\epsfysize4cm \epsfbox{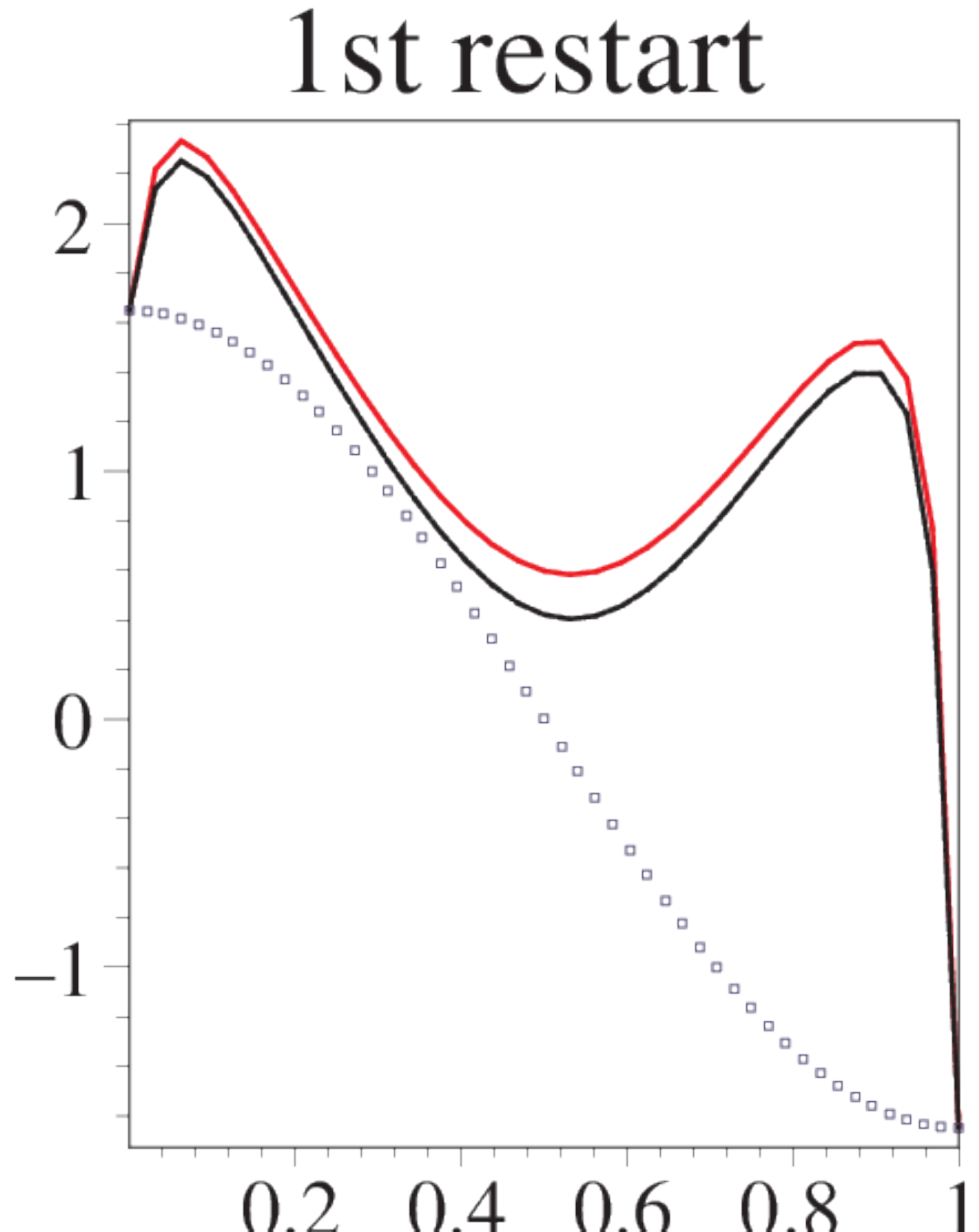} \ 
\epsfxsize3.3cm\epsfysize4cm \epsfbox{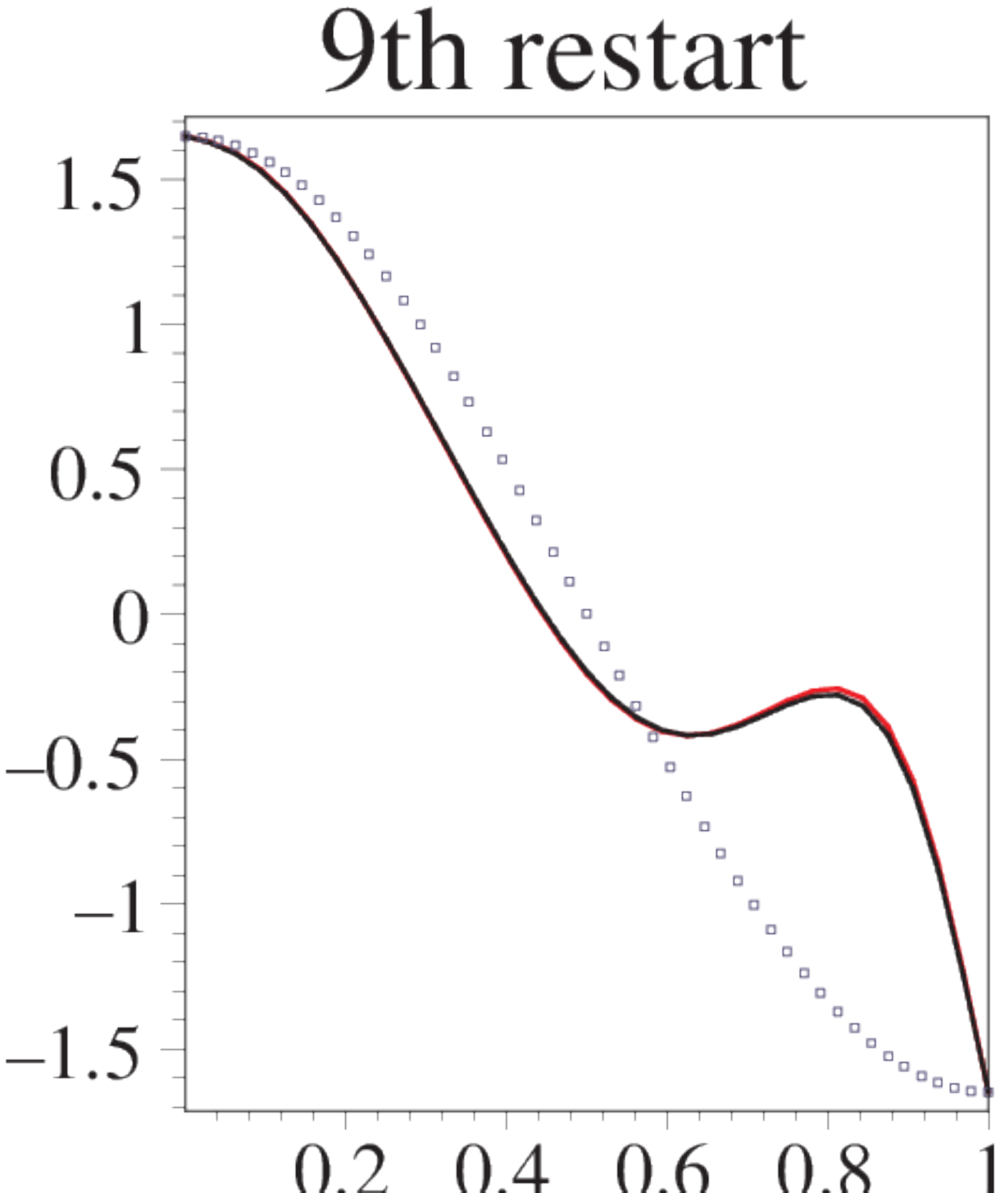} \ 
\epsfxsize3.3cm\epsfysize4cm \epsfbox{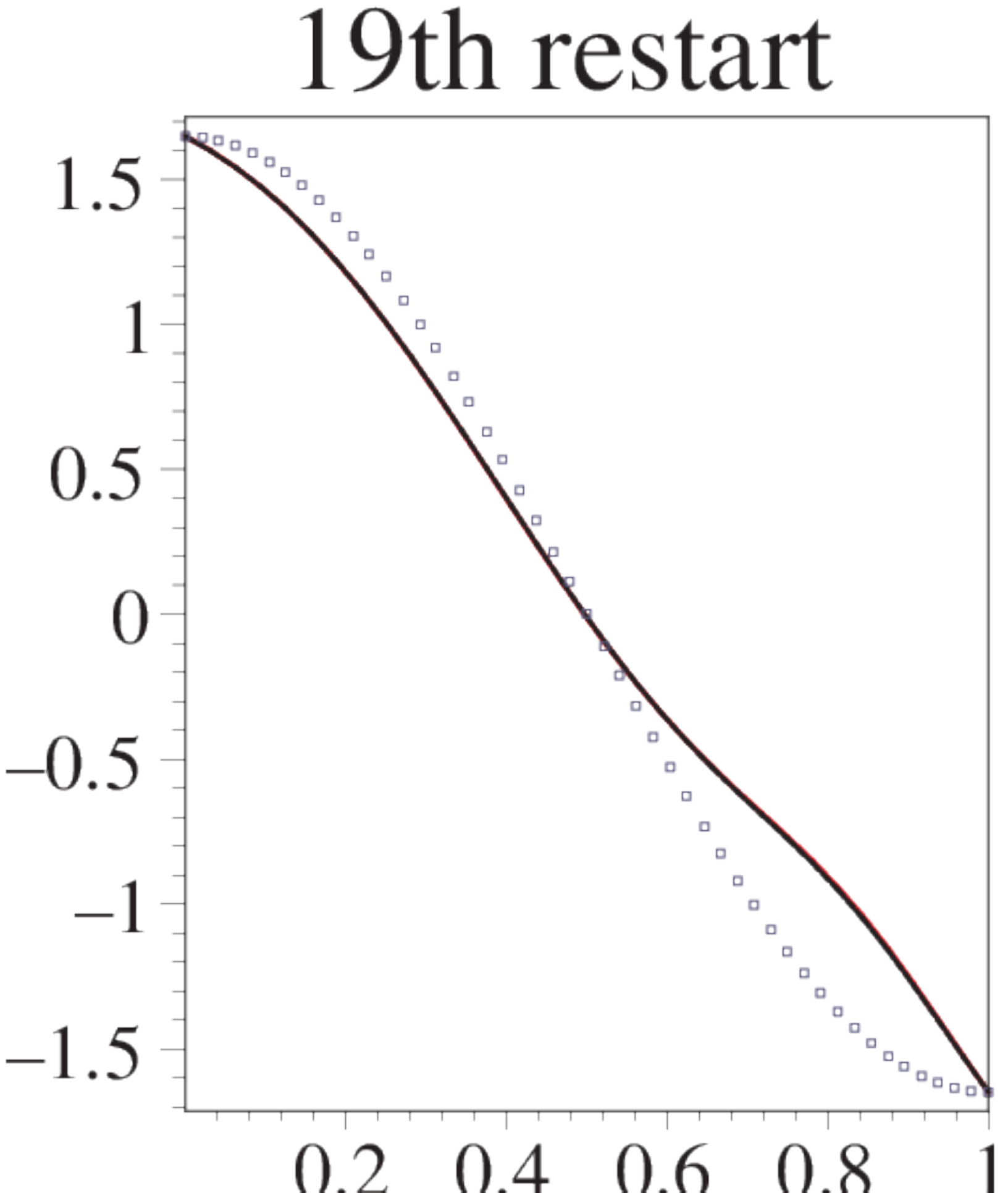} \ 
\epsfxsize3.3cm\epsfysize4cm \epsfbox{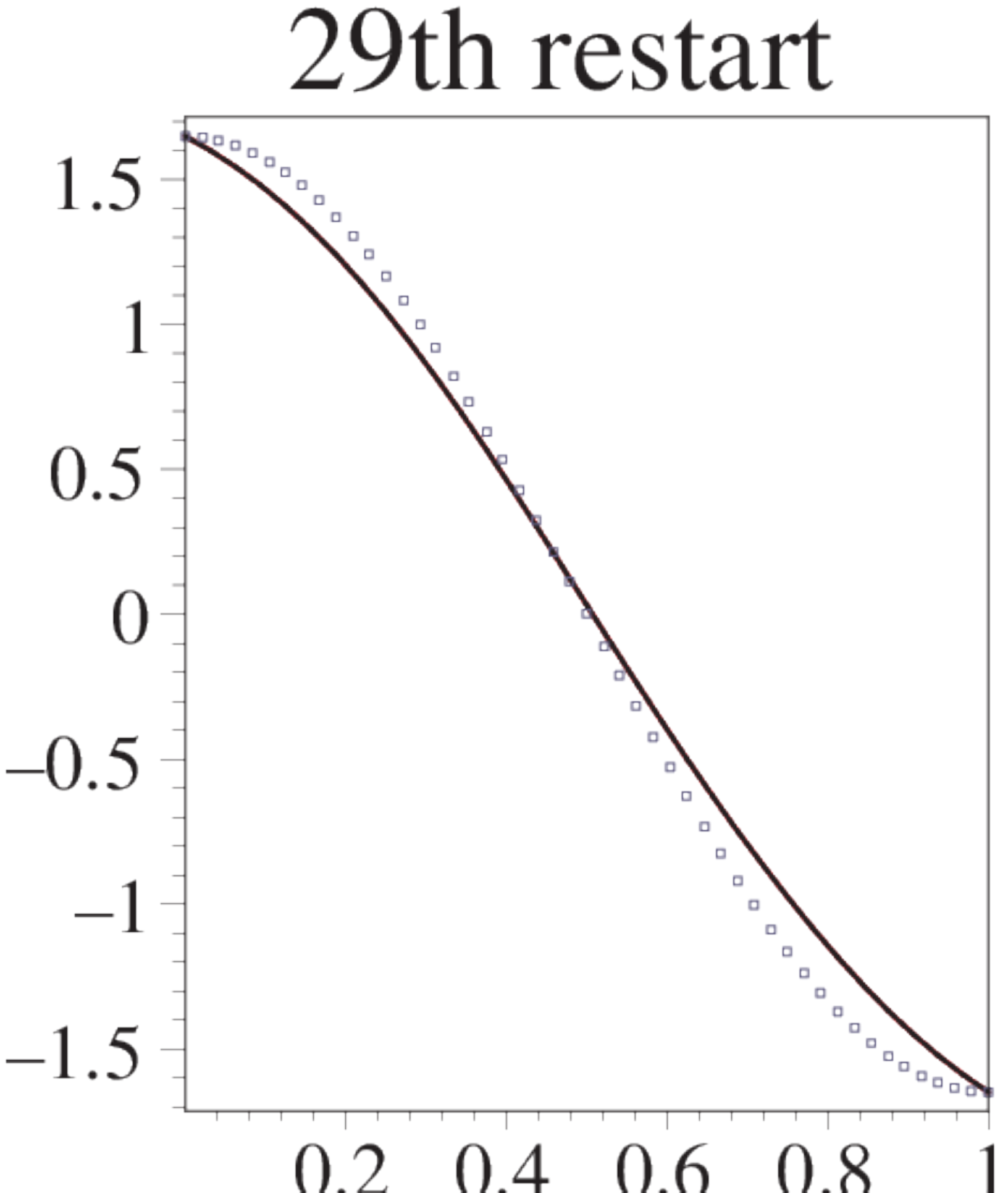} }
\end{picture}
\end{center} \vskip-0.6cm
\caption{Iteration with restart strategy (after every 50 steps) for a Cauchy 
problem with non harmonic solution \label{fig:ex2-f2}}
\end{figure}

Analogous as in the previous example, it is possible to accelerate the 
convergence of the iterative method by using a restart strategy. In 
Figure~\ref{fig:ex2-f2} we present the results obtained by restarting the 
iteration after every 50 steps (in order to obtain the results in Figure~%
\ref{fig:ex2-f2} we had to evaluate 100, 500, 1000 and 1500 iteration steps 
respectively).

A similar restart strategy was suggested in \cite{EnLe} (for the linear case) 
and produced nice results. It is worth mentioning that we have no analytical 
justification neither for the choice of the restart criterion nor for the 
improvement in the convergence rate.

\subsection{A problem with noisy data} \label{ssec:num-ex3}

For this exemple we consider once more the Cauchy problem described in 
Section~\ref{ssec:num-ex1}. To obtain the noisy Cauchy data we simply 
perturbed the exact Cauchy data
$$ (f,g) \ = \ \big( x^2 + 5x \, ,\ (1 + (x^2 + 5x)^2) \, (3x - 2) \big) $$
using an error of level 1\%.

\unitlength1mm
\begin{wrapfigure}{r}{8.3cm}
\begin{center} {\vspace{-4.6ex}
\begin{center}
\mbox{\epsfxsize4cm\epsfysize3cm \epsfbox{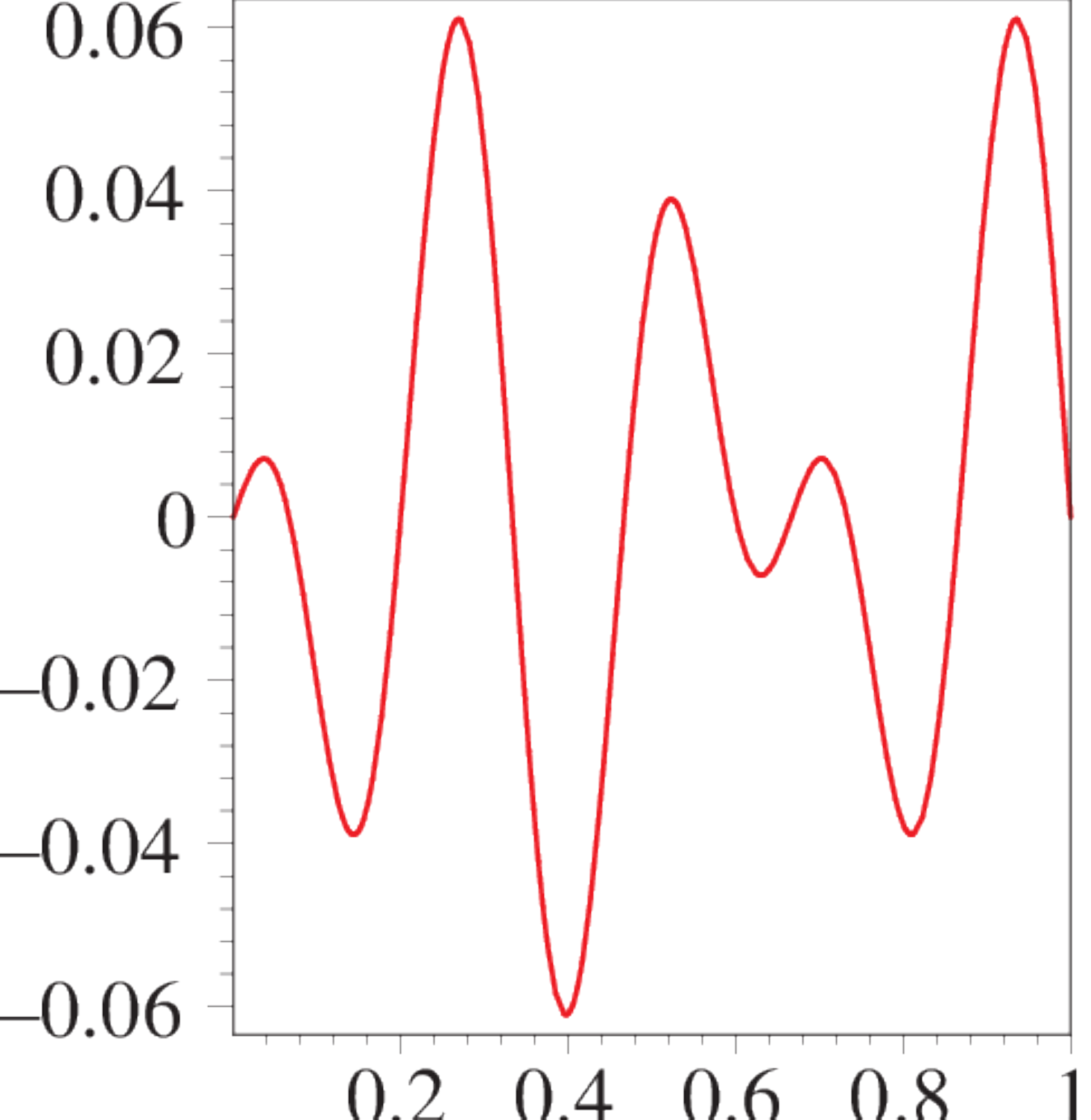}
      \epsfxsize4cm\epsfysize3cm \epsfbox{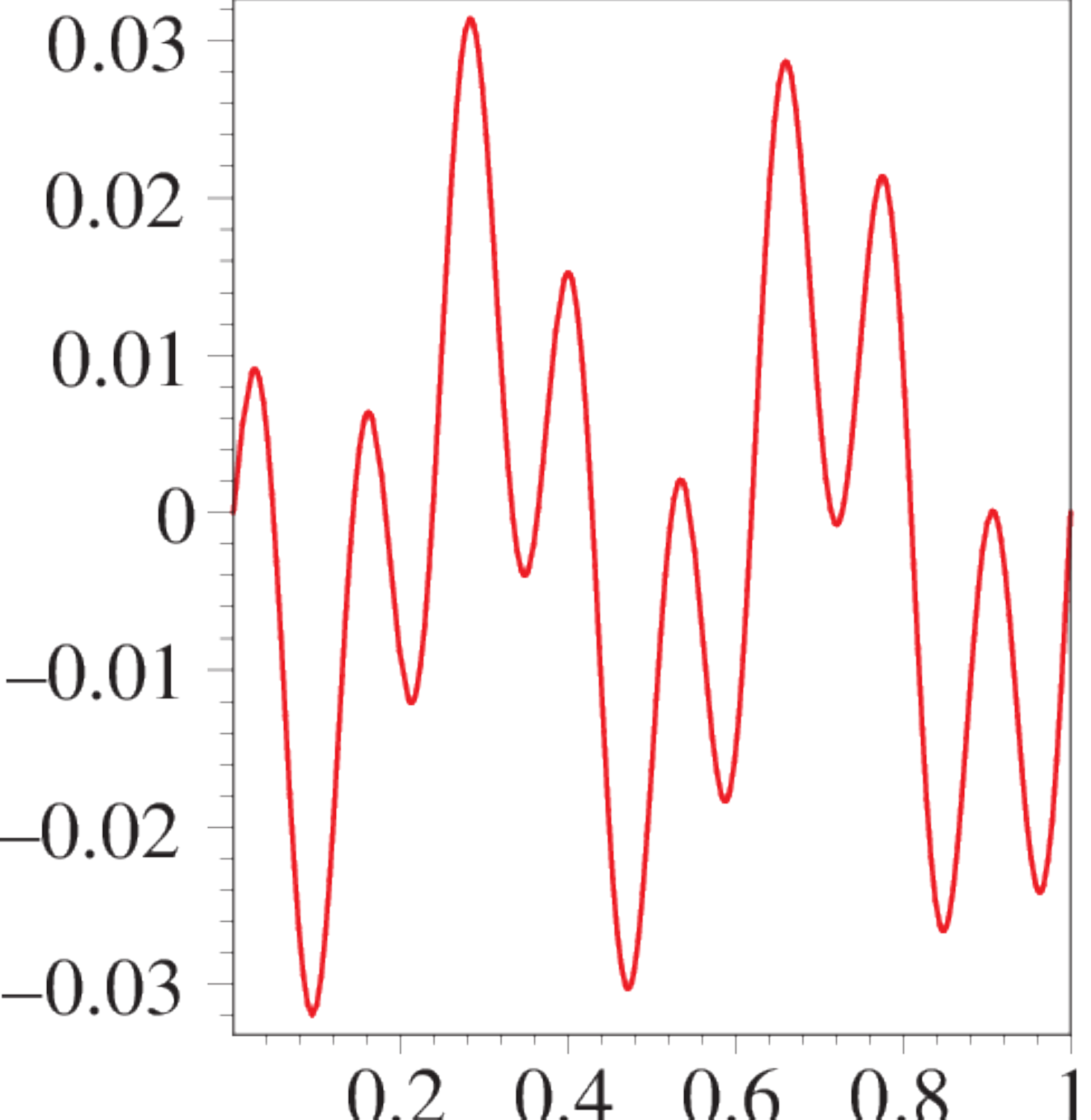} }
\end{center}
\centerline{\hspace{2cm} (a) \hspace{3.4cm} (b) \hfill \mbox{}}
\vspace{-0.5ex}
\caption{Generation of noisy data; (a) Perturbation added to the Dirichlet 
data; (b) Perturbation added to the Neumann data \label{fig:noisy-data} }
\vspace{-6ex}}
\end{center}
\end{wrapfigure}
In Figure~\ref{fig:noisy-data} we present the perturbations added to the 
exact Dirichlet and Neumann data.

For the iteration, we used again the Ces\`aro matrix $A$ in Section~%
\ref{ssec:num-ex1}. The initial guess, stopping rule and mesh refinement 
used for the computation are the same as those used in that section.

The numerical results corresponding to the Mann iteration are presented in 
Figure~\ref{fig:ex3-f1}: the dotted (blue) line corresponds to the exact 
solution; the dashed (black) line corresponds the iteration for exact Cauchy 
data (see Figure~\ref{fig:ex1-f1}); the solid (red) line corresponds to the 
iteration for the noisy data.

It is worth mentioning that our numerical solver was not able to handle with 
the non linear mixed boundary value problems when we tried to use a larger 
error level. To contour this problem we could alternatively refine our mesh 
or increase the maximum number of Newton iterations in the solver. However 
this would interfere with the comparison of results presented in Figure~%
\ref{fig:ex3-f1}.

\begin{figure}[t] \unitlength1cm
\begin{center}
\begin{picture}(14,4)
\centerline{
\epsfxsize3.3cm\epsfysize4cm \epsfbox{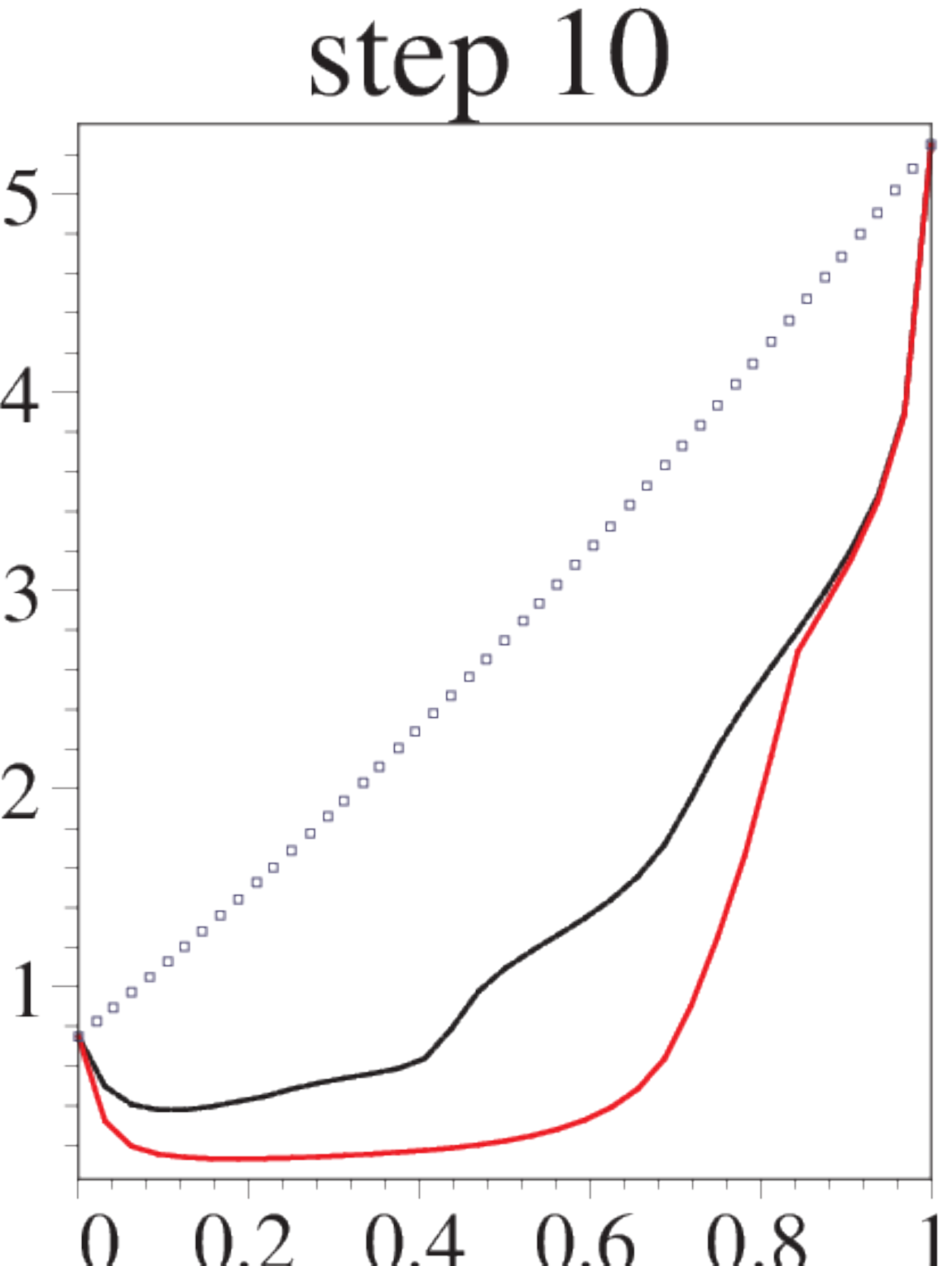} \ 
\epsfxsize3.3cm\epsfysize4cm \epsfbox{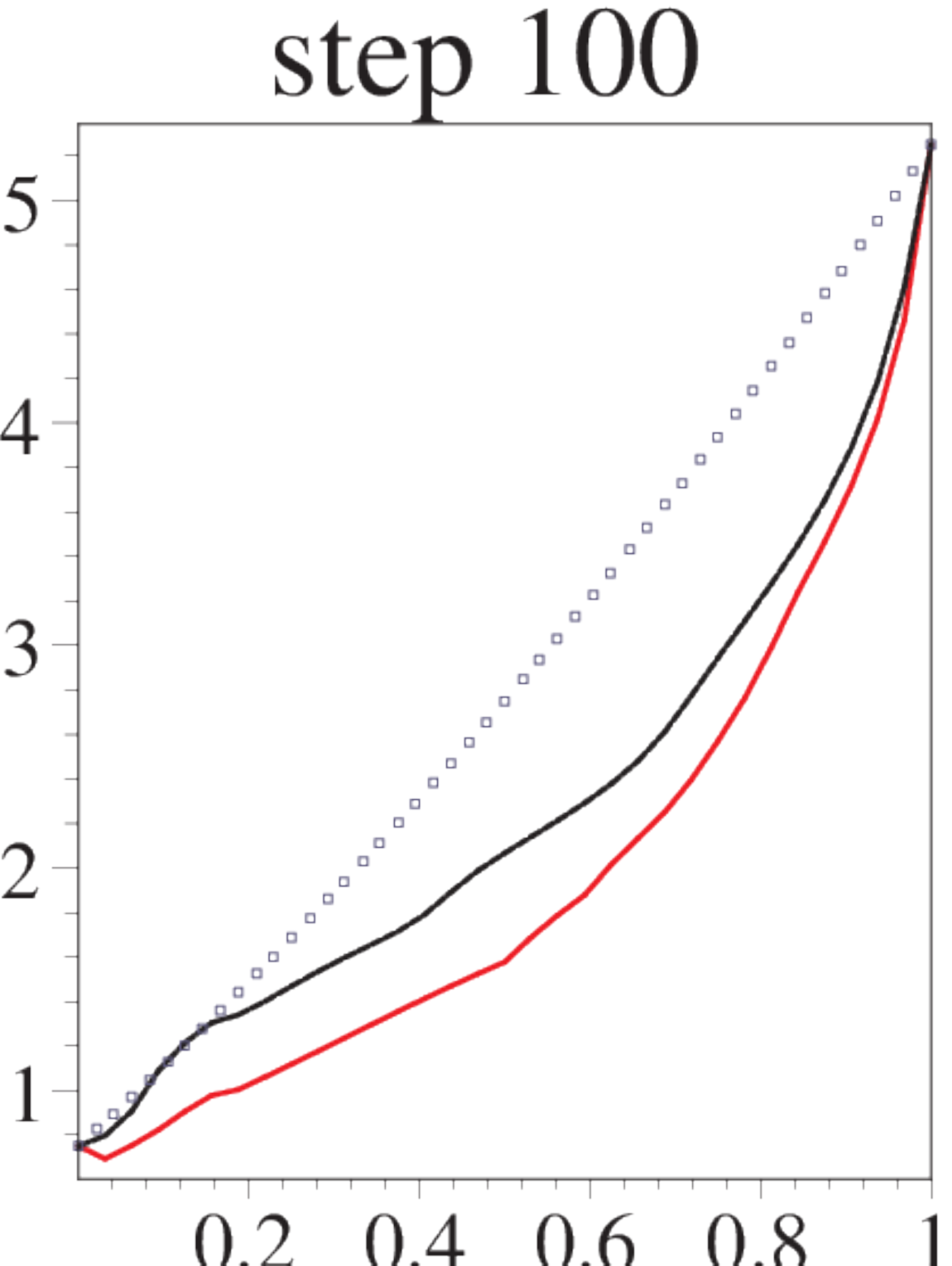} \ 
\epsfxsize3.3cm\epsfysize4cm \epsfbox{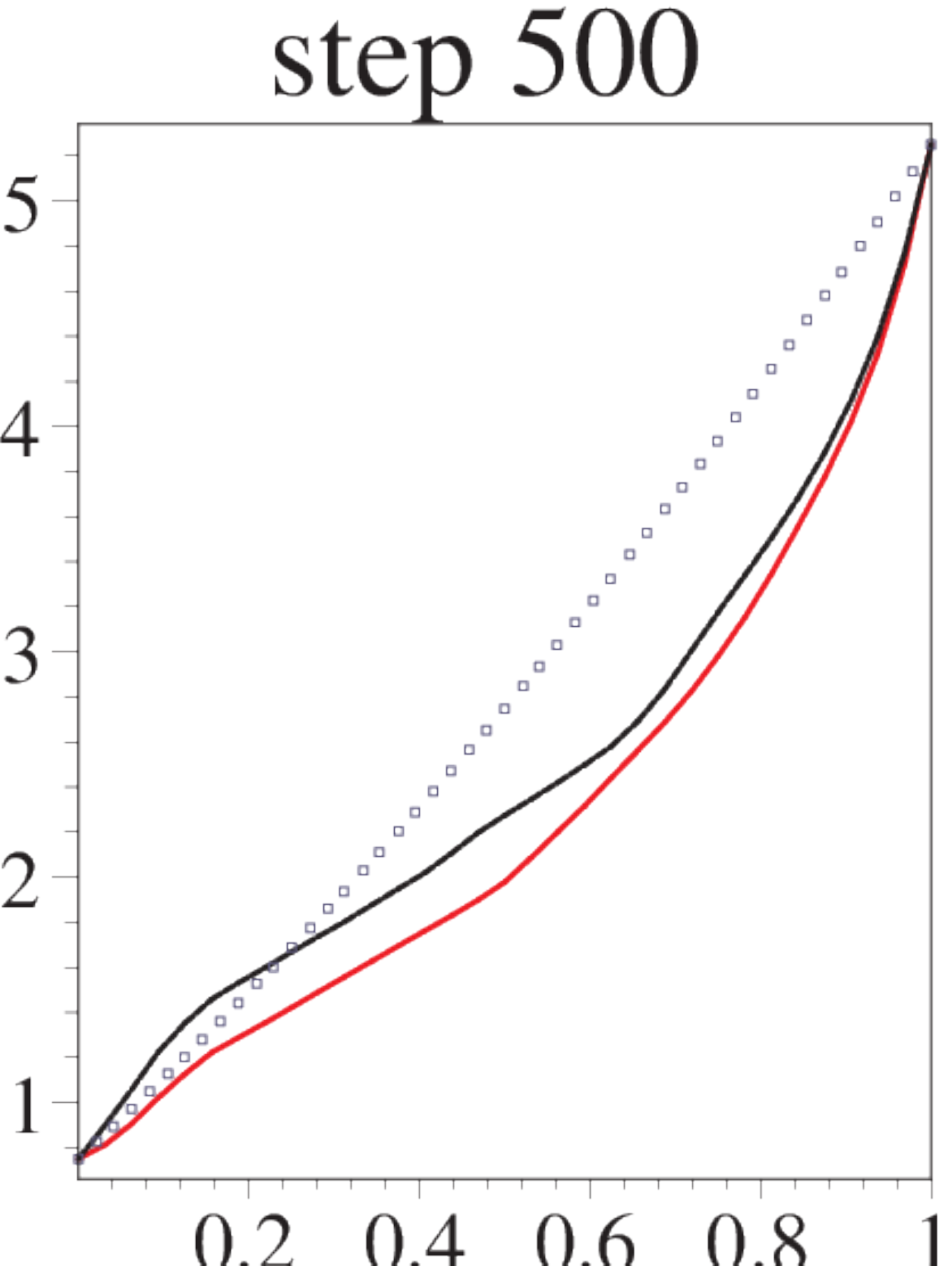} \ 
\epsfxsize3.3cm\epsfysize4cm \epsfbox{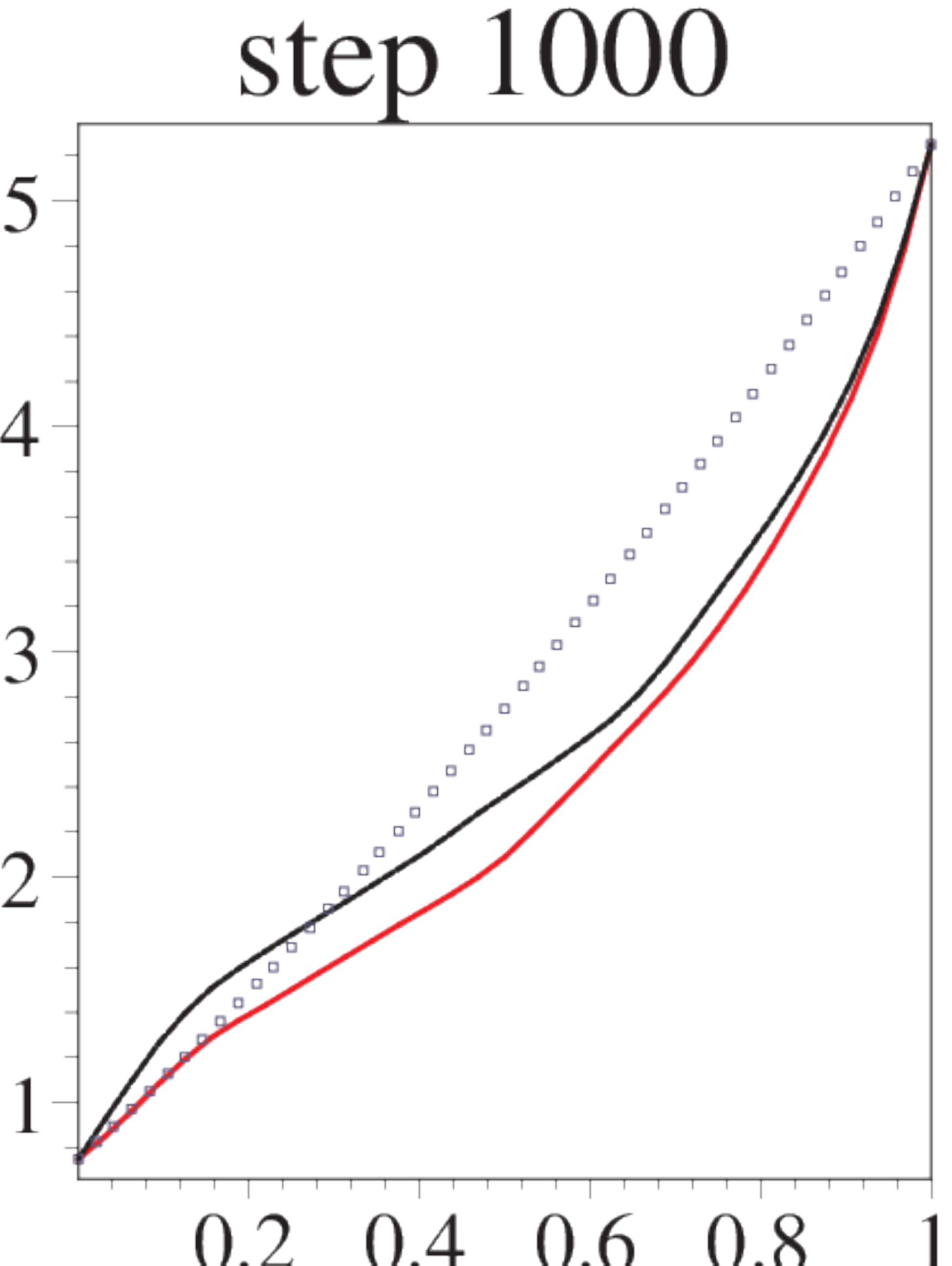} }
\end{picture}
\end{center} \vskip-0.6cm
\caption{Iteration for a Cauchy problem with noisy data \label{fig:ex3-f1}}
\end{figure}
%
%
%
\section{Final remarks} \label{sec:concl}

Let us assume $\partial\Omega = \Gamma_1 \cup \Gamma_2 \cup \Gamma_3$. Now, 
if we want to solve a Cauchy problem with data given at $\Gamma_1$ plus some 
further boundary condition (Neumann, Dirichlet, $\dots$) at $\Gamma_3$ (see 
the problems in Section~\ref{sec:numeric}). It is possible to adapt our 
iterative method for this type of problems. We just have to add the extra 
boundary condition at $\Gamma_3$ to both mixed boundary value problems at 
each iteration step. The over-determination of boundary data does not affect 
the analysis presented in this paper (see \cite{Le} for the linear case). 

The iterative methods proposed in Section~\ref{sec:it_cp} generate 
sequences of Neumann traces, which approximate the unknown Neumann boundary 
condition $q(u) u_{\nu|_{\Gamma_2}}$. Alternatively, we could define an 
iterative method, which produces a sequence of Dirichlet traces (see the 
problems in Section~\ref{sec:numeric}). This was already suggested in the 
linear case in \cite{Le}). The convergence proof for this iterative method 
is quite similar to the one discussed in this paper.

\subsection*{Acknowledgment}

The authors would like to thank Professor H.W. Engl (Institut f\"ur 
Industriemathe\-matik, Kepler Universit\"at Linz) for several useful 
comments and discussions.
%
%
%
%
%
%
%
%
%
%

\newpage

\end{document}